\documentclass[11pt]{elsarticle}

\usepackage{lineno}
\modulolinenumbers[1]

\usepackage{amssymb,amsmath,amsthm}
\usepackage{indentfirst}
\usepackage{caption}
\usepackage{enumerate}
\usepackage{mathrsfs}
\usepackage{subcaption}

\usepackage{amssymb}
\usepackage{amsmath}
\usepackage{amsthm}
\usepackage{tikz}
\usepackage{pgfplots}
\pgfplotsset{compat=1.18}
\usetikzlibrary{patterns}

\usepackage{xcolor}

\definecolor{c1}{cmyk}{0.3, 0.1, 0.0, 0.2}

\newtheorem{theorem}{Theorem}[section]

\newtheorem{lemma}[theorem]{Lemma}

\allowdisplaybreaks[4]

\usepackage{lineno}

\journal{}

\begin{document}

\begin{frontmatter}



\title{Hopf bifurcation in a memory-based diffusion competition model with spatial heterogeneity} 
\tnotetext[mytitlenote]{This research is supported by the National Natural Science Foundation of China (No. 12271525).}

\author{Shu Li\fnref{label1}}
\ead{shli97@csu.edu.cn}

\author{Binxiang Dai\fnref{label1}\corref{mycorrespondingauthor}}
\cortext[mycorrespondingauthor]{Corresponding author}
\ead{bxdai@csu.edu.cn}

\address[label1]{School of Mathematics and Statistics, HNP-LAMA, Central South University, Changsha, Hunan, 410083, PR China}

\begin{abstract}
In this paper, we investigate a Lotka-Volterra competition-diffusion system with self-memory effects and spatial heterogeneity under Dirichlet boundary conditions. We focus on how memory strength influences the coexistence and stability of competing species. By analyzing the characteristic equation, we establish the existence and stability of a spatially nonhomogeneous positive steady state and demonstrate the occurrence of Hopf bifurcation as memory delay increases. Our results reveal that both weak and some opposing memory effects of two competing species promote stable coexistence, while strong memory may destabilize the system and lead to periodic oscillations. Spatial heterogeneity further enriches the dynamical behaviors. Numerical simulations are presented to confirm the theoretical results.
\end{abstract}


\end{frontmatter}

\section{Introduction}
Understanding the spatial dynamics of interacting biological species is a fundamental topic in mathematical ecology. Reaction-diffusion systems have long been employed to model the spread and interaction of populations in heterogeneous environments, capturing essential features such as diffusion-driven instabilities, pattern formation, and coexistence mechanisms \cite{BH1996, W1996}. In recent years, increasing attention has been paid to the role of cognition, including memory and perception, in shaping animal movement and spatial distribution \cite{FL2013, WS2023}. In particular, Shi \cite{SW2020} incorporated a modified Fickian diffusion into the reaction-diffusion model to describe episodic-like memory and derived the following model:
\begin{equation}\label{i1}
	\begin{cases}
		u_t=d_1\triangle u(x,t)+d_2 div (u(x,t)\nabla u(x,t-\tau))+f(u(x,t)), & x\in\Omega, t>0, \\
		\partial_{\vec{n}} u=0, & x\in\partial\Omega, t>0, 
	\end{cases}
\end{equation}
where delay $\tau>0$ is the average memory period; $d_1>0$ and $d_2\in\mathbb R$ represent the random and memory-based diffusion coefficients, respectively. When $d_2>0$, memory induces animals to exhibit repulsive movements away from past locations, while when $d_2<0$,  it causes them to display attractive movements towards these past locations. Building upon model \eqref{i1}, many researchers have investigated how different types of memory formulations, boundary conditions, and reaction terms affect the dynamical behavior of the system. For instance, under Neumann boundary conditions, various memory effects have been studied in \cite{SW2019, SWW2019, WF2022}, while similar analyses under Dirichlet boundary conditions can be found in \cite{AW2020}. Regarding two-species interaction systems, most existing works focus on predator-prey models, exploring the influence of memory-based diffusion \cite{LL2022, LJ2024, LY2024, LW2023, SP2021} and therein. Wang et al. \cite{WW2023} examined a competition model with discrete memory under Neumann boundary conditions as follows:
\begin{equation} \label{i2}
	\begin{cases}
		u_t=d_1\triangle u+d_{11} \nabla\cdot(u\nabla u_{\tau})+d_{12} \nabla\cdot(u\nabla v_{\tau})+u(m(x)-u-bv), & x\in\Omega,t>0,  \\
		v_t=d_2\triangle v+d_{21} \nabla\cdot(v\nabla u_{\tau})+d_{22} \nabla\cdot(v\nabla v_{\tau})+v(m(x)-cu-v), & x\in\Omega,t>0,  \\
		\partial_{\nu}u=\partial_{\nu}v=0 , & x\in\partial\Omega,t>0,  \\
		u(x,t)=u_0(x,t), v(x,t)=v_0(x,t), & x\in\Omega,t\in (-r,0].
	\end{cases}
\end{equation}
Their results indicate that, under Neumann boundary conditions, the interplay between identical resource functions and memory-based diffusion, including self-memory and cross-memory, can stabilize nonconstant positive steady states in the weak competition condition. This highlights the crucial role of two types of memory effects in driving and maintaining spatially heterogeneous coexistence patterns. However, our previous studies have shown that self-memory may help a species gain an advantage in competition, while cross-memory can increase the risk of extinction for competing species. Therefore, it is of particular interest to investigate how different intensities of self-memory in two competing species affect their competitive outcomes. In this paper, we consider the following model:
\begin{equation}\label{1}
	\begin{cases}
		u_t(x,t)=\triangle u(x,t)+d_1 \nabla\cdot(u(x,t)\nabla u(x,t-\tau))+ \lambda_1 u(x,t)(r_1(x)-a_{11}u(x,t)-a_{12}v(x,t)), & x\in\Omega,t>0, \\
		v_t(x,t)=\triangle v(x,t)+d_2 \nabla\cdot(v(x,t)\nabla v(x,t-\tau))+ \lambda_2 v(x,t)(r_2(x)-a_{21}u(x,t)-a_{22}v(x,t)), & x\in\Omega,t>0,  \\
		u(x,t)=v(x,t)=0, & x\in\partial\Omega,t>0,  \\
	\end{cases}
\end{equation}
where \( u = u(x,t) \) and \( v = v(x,t) \) denote the population densities of two competing species at position \( x \) and time \( t \), respectively. The parameters \( d_1 \) and \( d_2 \) represent the ratios of memory-based diffusion to random diffusion for species \( u \) and \( v \), respectively. The coefficients \( a_{12} \) and \( a_{21} \) describe the interspecific competition intensities related to species \( u \) and \( v \), while \( a_{11} \) and \( a_{22} \) denote the intraspecific self-limitation rates. We make the following assumptions on the resource functions \( r_1(x) \) and \( r_2(x) \) of species \( u \) and \( v \), respectively:
\begin{enumerate}
	\item[$(\bf{H_0})$] $r_i(x)\in C^{\alpha}(\bar \Omega)$ $(\alpha\in(0, 1))$, $\max\limits_{x\in\bar\Omega} r_i(x)>0$, $i=1,2$. 
\end{enumerate}

In the following analysis, we investigate the stability and existence of the Hopf bifurcation to model \eqref{1}. Our main results can be summarized as follows:
\begin{itemize}
	\item There is a positive steady state $(u_s, v_s)$ of model \eqref{1} (see Theorem \ref{thm2.1});
	\item If $(d_1, d_2)\in D_2$, then the positive steady-state solution $(u_s, v_s)$ is locally asymptotically stable when $\tau\ge 0$ (see Theorem \ref{thm9});
	\item Under some additional conditions, there exists a sequence of values $\{\tau_n\}_{n=0}^{\infty}$, such that the positive steady state $(u_s, v_s)$ is locally asymptotically stable when $\tau\in[0, \tau_0)$, and unstable when $\tau\in(\tau_0, \infty)$; Moreover, a Hopf bifurcation occurs at $(u_s, v_s)$ when $\tau=\tau_n$, $n=0, 1, 2, \dots$ (see Theorem \ref{thm7}, \ref{thm9}, and \ref{th3.9}).
\end{itemize}

In model \eqref{1}, when \( d_1 = d_2 = 0 \), \eqref{1} reduces to the classical reaction-diffusion competition model in a spatially heterogeneous environment, which has been extensively studied. For instance, Hastings \cite{has} explored a two-species competition model in a spatially heterogeneous yet temporally constant environment, demonstrating that a species with a slower diffusion rate has a competitive advantage when introduced at low density. Similar conclusions were reached by Dockery et al. \cite{doc}, who emphasized the role of diffusion rate in invasion success. He and Ni \cite{HN, HN2, HN3} considered the competition model with spatial heterogeneity under the Neumann boundary conditions; they provided a complete classification of the global dynamics of the system. Their results indicate that the system either admits a globally asymptotically stable semi-trivial steady state, or has a unique globally asymptotically stable coexistence steady state, or enters a degenerate case, where a compact global attractor exists and consists of a continuum of steady states connecting two semi-trivial steady states. 

The rest of our paper is organized as follows. In Section 2, the existence and the explicit expression of the nonconstant steady state solution of \eqref{1} is given. In Section 3, by analyzing the corresponding eigenvalue problem, the stability and the existence of Hopf bifurcation of \eqref{1} are obtained. In section 4, numerical simulations are performed to verify the theoretical results.

Throughout this paper, denote $X=H^2(\Omega)\cap H_0^1(\Omega)$ and $Y=L^2(\Omega)$, with $H_0^1(\Omega)=\{u\in H^1(\Omega)|u(x)=0, \forall x\in\partial\Omega\}$. For any subspace $Z$, we define the Complex valued space of $Z$ is $Z_{\mathbb C}:=Z\bigoplus iZ=\{x_1+ix_2|x_1, x_2\in Z\}$. Denote $C((-\infty,0],Y)$ is a Banach space that maps continuously from $(-\infty,0]$ to $Y$. For a Complex-valued Hilbert space $Y_{\mathbb C}$, the inner product is defined by $\langle u,v\rangle=\int_\Omega \Bar{u}(x)v(x)\mathrm{d}x$.

\section{Existence of the positive steady state solutions}

In this section, we investigate the existence of non-constant steady-state solutions of \eqref{1}, which are solutions of the following corresponding elliptic system:
\begin{equation}\label{2.2}
\begin{cases}
\triangle u(x)+d_1 \nabla\cdot(u(x)\nabla u(x)+ \lambda_1 u(x)(r_1(x)-a_{11}u(x)-a_{12}v(x))=0, & x\in\Omega, \\
\triangle v(x)+d_2 \nabla\cdot(v(x)\nabla v(x)+ \lambda_2 v(x)(r_2(x)-a_{21}u(x)-a_{22}v(x))=0, & x\in\Omega, \\
u(x)=v(x)=0, & x\in\partial\Omega.  \\
\end{cases}
\end{equation}
Define a nonlinear operator $F: X^2\times \mathbb R^2\to Y^2$ by
\[
F(U,\lambda)=\left(\begin{array}{cc}
	\triangle u+d_1\nabla\cdot(u\nabla u)+\lambda_1 u(r_1(x)-a_{11}u-a_{12}v)\\
	\triangle v+d_2\nabla\cdot(v\nabla v)+\lambda_2 v(r_2(x)-a_{21}u-a_{22}v)
\end{array}\right),
\]
with $U=(u,v)^T\in X^2$ and $\lambda=(\lambda_1, \lambda_2)^T\in\mathbb R^2$. For any fixed $\lambda$, $U_0=(0,0)$ is always a trivial steady state solution of \eqref{2.2}. The $Frech\Acute{e} t$ derivative of $F(U,\lambda)$  with respect to $U$ evaluated at $(U_0,\lambda_*)$ is
\[
\mathcal{L}_{\lambda}:=D_U F(0,\lambda_*)=\left(\begin{array}{cc}
	\triangle+\lambda_1r_1(x) & 0\\
	0 & \triangle+\lambda_2r_2(x)
\end{array}\right),
\]
let $\lambda_*=(\lambda_1^*, \lambda_2^*)\in\mathbb R^2$, where $\lambda_1^*$, $\lambda_2^*$ are the principle eigenvalues of the following eigenvalue problems:
\begin{equation}\label{2.3}
	\begin{cases}
		\triangle \phi+\lambda_1 r_1(x)\phi=0, & x\in\Omega, \\
		\phi(x)=0, & x\in\partial\Omega, 
	\end{cases}
\end{equation}
and 
\begin{equation}\label{2.4}
	\begin{cases}
		\triangle \varphi+\lambda_2 r_2(x)\varphi=0, & x\in\Omega, \\
		\varphi(x)=0, & x\in\partial\Omega, 
	\end{cases}
\end{equation}
respectively. By the variational method, $\lambda_1^*$ and $\lambda_2^*$ can be represented as the following form:
\begin{equation}\label{3.5}
	\lambda_1^*=\inf\limits_{\phi\in H_1(\Omega), \int_{\Omega}r_1(x)\phi^2\mathrm{d}x>0} \dfrac{\int_{\Omega}|\nabla\phi |^2\mathrm{d}x}{\int_{\Omega}r_1(x)\phi^2\mathrm{d}x}, \quad 
	\lambda_2^*=\inf\limits_{\varphi\in H_1(\Omega), \int_{\Omega}r_2(x)\varphi^2\mathrm{d}x>0} \dfrac{\int_{\Omega}|\nabla\varphi |^2\mathrm{d}x}{\int_{\Omega}r_2(x)\varphi^2\mathrm{d}x},
\end{equation}
with the corresponding principal eigenfunctions $\phi_*>0$ and $\psi_*>0$, respectively. Without loss of generality, we assume that $\|\phi_*\|_{Y_{\mathbb C}}^2=1$ and $\|\psi_*\|_{Y_{\mathbb C}}^2=1$. Then let 
\[
\mathcal N(\mathcal L_{\lambda_*})=\mathrm{span}\{\Phi_*, \Psi_*\}, 
\]
where $\Phi_*=(\phi_*, 0)^T$, $\Psi_*=(0, \psi_*)^T$. To find the steady-state solution of \eqref{2.2}, we decompose the spaces as follows:
\[
X^2=\mathcal{N}(\mathcal L_{\lambda_*})\oplus X_1^2, \quad
Y^2=\mathcal{N}(\mathcal L_{\lambda_*})\oplus Y_1^2,
\]
where
\[
X_1^2=\{y\in X^2: \langle\Phi_*, y\rangle=\langle\Psi_*, y\rangle =0\}, 
\]
and
\[
Y_1^2=\mathcal{R}(\mathcal L_{\lambda_*})=\{y\in Y^2: \langle\Phi_*, y\rangle=\langle\Psi_*, y\rangle =0\}.
\]

For simplicity, we denote
\begin{align}\label{kappa}
	& \kappa_1=\cos{\omega}\int_{\Omega}\phi_*\nabla\cdot(\phi_*\nabla\phi_*)\mathrm{d}x, \quad
	\kappa_2=\sin{\omega}\int_{\Omega}\psi_*\nabla\cdot(\psi_*\nabla\psi_*)\mathrm{d}x, \nonumber\\
	& \kappa_3=\int_{\Omega}\cos{\omega}\phi_*^3\mathrm{d}x>0, \quad
	\kappa_4=\int_{\Omega}\sin{\omega}\phi_*^2\psi_*\mathrm{d}x>0, \quad
	\kappa_5=\int_{\Omega}\cos{\omega}\phi_*\psi_*^2\mathrm{d}x>0, \nonumber\\
	&\kappa_6=\int_{\Omega}\sin{\omega}\psi_*^3\mathrm{d}x>0, \quad
	 \kappa_7=\int_{\Omega}\cos{\omega}\phi_*^2\psi_*\mathrm{d}x>0, \quad
	\kappa_8=\int_{\Omega}\sin{\omega}\phi_*\psi_*^2\mathrm{d}x>0.
\end{align}
Then, we obtain the following theorem on the existence of positive steady-state solutions of \eqref{1}.
\begin{theorem}\label{thm2.1}
	Asuume that $(\bf{H_0})$ holds, there exists a constant $\delta>0$ and a continuously differentiable map $s\mapsto(\lambda_1(s), \lambda_2(s), (w_1(s), w_2(s)))$ from $[0, \delta]$ to $\mathbb R^2 \times X_1$, such that for any $s\in[0, \delta]$ and $\omega\in(0, \frac{\pi}{2})$, \eqref{1} has a positive steady-state solution as the following form:
\begin{equation}\label{2.6} 
\begin{cases}
	u_s=s[\cos{\omega}\phi_*+w_1(s)], \\
	v_s=s[\sin{\omega}\psi_*+w_2(s)].
\end{cases}  
\end{equation}
Moreover, $(\lambda_1(s), \lambda_2(s), (w_1(s), w_2(s)))$ satisfies 
\begin{equation}\label{2.7}
	\begin{cases}
		\lambda_1(s)=\lambda_1^*+\lambda_1'(0)s+o(s), \\
		\lambda_2(s)=\lambda_2^*+\lambda_2'(0)s+o(s), \\
		(w_1(0), w_2(0))^T=(0, 0)^T,
	\end{cases}
\end{equation}
where 
\begin{equation}\label{2.8}
	\begin{cases}
	\lambda_1'(0)=\dfrac{\lambda_1^*(a_{11}\kappa_3+a_{12}\kappa_4)-d_1\kappa_1}{\int_{\Omega}r_1(x)\phi_*^2\mathrm{d}x}, \\
	\lambda_2'(0)=\dfrac{\lambda_2^*(a_{21}\kappa_5+a_{22}\kappa_6)-d_2\kappa_2}{\int_{\Omega}r_2(x)\psi_*^2\mathrm{d}x}. 
	\end{cases}
\end{equation}
\end{theorem}

\begin{proof}
	We assume there exists a solution of \eqref{2.2} in the following form:
	\begin{equation}\label{2.9}
	u=s(\cos{\omega}\phi_*+w_1), \quad
	v=s(\sin{\omega}\psi_*+w_2),
	\end{equation}
where $s\in(0,\delta)$, $\omega\in(0,\frac{\pi}{2})$. Substituting \eqref{2.9} into \eqref{2.2}, we see that $(\lambda_1, \lambda_2, w_1, w_2, s)$ satisfies:
\[
\mathcal F(\lambda_1, \lambda_2, w_1, w_2, s)=(\mathcal L_{\lambda_*}+\mathcal D)\left(\begin{array}{cc}
	\cos{\omega}\phi_*+w_1 \\ \sin{\omega}\psi_*+w_2
\end{array}\right)-s\left(\begin{array}{cc}
	\lambda_1 a_{11}(\cos{\omega}\phi_*+w_1)+\lambda_1 a_{12}(\sin{\omega}\psi_*+w_2) \\ \lambda_2 a_{21}(\cos{\omega}\phi_*+w_1)+\lambda_2 a_{22}(\sin{\omega}\psi_*+w_2)
\end{array}\right),
\]
where 
\[ 
\mathcal D\left(\begin{array}{cc}
	\cos{\omega}\phi_*+w_1 \\ \sin{\omega}\psi_*+w_2
\end{array}\right)=\left(\begin{array}{cc}
	d_1\nabla\cdot((\cos{\omega}\phi_*+w_1)\nabla (s(\cos{\omega}\phi_*+w_1)) \\ d_2\nabla\cdot((\sin{\omega}\psi_*+w_2)\nabla (s(\sin{\omega}\psi_*+w_2))
\end{array}\right),
\]
it is clear that $\mathcal F(\lambda_1, \lambda_2, 0, 0, 0)=0$. Taking the Fr$\Acute{e}$chet derivative  of $\mathcal F (\lambda_1, \lambda_2, w_1, w_2, s)$ with respect to $(\lambda_1, \lambda_2, w_1, w_2)$ at $(\lambda_1^*, \lambda_2^*, 0, 0)$, we obtain that
\begin{align*}
	& D_{(\lambda_1, \lambda_2, w_1, w_2)}\mathcal F (\lambda_1^*, \lambda_2^*, 0, 0, 0)[\epsilon, \varrho, \zeta, \chi] = \mathcal L_{\lambda_*}
\left(\begin{array}{cc}
	\zeta \\ \chi
\end{array}\right)+
\left(\begin{array}{cc}
	\epsilon r_1(x)\cos{\omega}\phi_* \\ \varrho r_2(x)\sin{\omega}\psi_*
\end{array}\right),
\end{align*}
which is a linear mapping. Similiar to the proof of \cite{LJ2024}, one can easily check that $D_{(\xi, \eta, \alpha, \beta)}\mathcal F (\xi_{\lambda_*}, \eta_{\lambda_*}, \alpha_{\lambda_*}, \beta_{\lambda_*}, \lambda_*)$ is bijective from $X_1^2\times\mathbb R^2$ to $Y^2$. It follows from the implicit function theorem that there exists a constant $\delta>0$ and a continuously differentiable map $s \mapsto(\lambda_1(s), \lambda_2(s), (w_1(s), w_2(s))): [0,\delta]\to \mathbb R^2\times X_1^2$ such that $(\lambda_1(s), \lambda_2(s), (w_1(s), w_2(s)))$ satisfies \eqref{2.7} and 
\begin{equation}\label{2.10}
	\mathcal F(\lambda_1(s), \lambda_2(s), w_1(s), w_2(s), s)=0, \quad s\in[0, \delta].
\end{equation}
It is clear that \eqref{2.6} is a positive steady-state solution of \eqref{1}.

Moreover, Taking the derivative of \eqref{2.10} with respect to $s$ evaluated at $s=0$ yields
\begin{equation}\label{2.11}
	\begin{cases}
		0= & \triangle w_1'(0)+d_1\nabla\cdot((\cos{\omega}\phi_*)\nabla(\cos{\omega}\phi_*))+\lambda_1'(0)\cos{\omega}\phi_*r_1(x) \\
		& + \lambda_1^*[w_1'(0)r_1(x)-\cos{\omega}\phi_*(a_{11}\cos{\omega}\phi_*+a_{12}\sin{\omega}\psi_*)],\\
		0= & \triangle w_2'(0)+d_2\nabla\cdot((\sin{\omega}\psi_*)\nabla(\sin{\omega}\psi_*))+\lambda_2'(0)\sin{\omega}\psi_*r_2(x) \\
		& + \lambda_2^*[w_2'(0)r_2(x)-\sin{\omega}\psi_*(a_{21}\cos{\omega}\phi_*+a_{22}\sin{\omega}\psi_*)].
	\end{cases}
\end{equation}
Multiplying the first and second equation of \eqref{2.11} with $\phi_*$ and $\psi_*$, respectively, and then integrating over $\Omega$, we obtain \eqref{2.8}. The proof is completed.
\end{proof}

\section{Stability and Hopf bifurcation}
In this section, we will study the stability of the solution $(u_s,v_s)$ defined in \eqref{2.6}, which is a nonconstant steady-state solution of \eqref{1}. In the following, we denote $\lambda_{1s}:=\lambda_1(s)$ and $\lambda_{2s}:=\lambda_2(s)$ for $s\in[0, \delta]$. Then linearizing the system \eqref{1} at $(u_s, v_s)$, we obtain:
\begin{equation}\label{3.12}
	\begin{cases}
	u_t= \triangle u+d_1\nabla\cdot[u\nabla u_s]+d_1\nabla\cdot[u_s\nabla u(x,t-\tau)] +\lambda_{1s} u(r_1(x)-2a_{11}u_s-a_{12}v_s)-a_{12}\lambda_{1s} u_sv,  &  x\in\Omega, t>0 \\
	v_t= \triangle v+d_2\nabla\cdot[v\nabla v_s]+d_2\nabla\cdot[v_s\nabla v(x,t-\tau)] -a_{21}\lambda_{2s} v_su+\lambda_{2s} v(r_2(x)-a_{21}u_s-2a_{22}v_s),  &  x\in\Omega, t>0 \\
	u(x,t)=v(x,t)=0, & x\in\partial\Omega, t>0.
	\end{cases}
\end{equation}

From \cite{SWW2021}, the semigroup induced by the solutions of \eqref{3.12} has the infinitesimal generator $A_{\tau}(s)$ which is defined by:
\begin{equation}\label{3.14}
	A_{\tau}(s)\Phi=\dot{\Phi},
\end{equation}
with the domain
\[
\mathcal D(A_{\tau}(s))=\lbrace \Phi\in C_{\mathbb C}\cap C_{\mathbb C}^1: \dot{\Phi}(0)= A(s)\Phi+B(s)\Phi(-\tau) \rbrace,
\]
where $\Phi=(\Phi_1, \Phi_2)^T\in X_{\mathbb C}^2$, $C_{\mathbb C}=C((-\infty,0], Y_{\mathbb C}^2)$ and  $C_{\mathbb C}^1=C^1((-\infty,0], Y_{\mathbb C}^2)$, the linear operators $A(s):\mathcal D(A(s))\to Y_{\mathbb C}^2$, $B(s): Y_{\mathbb C}^2\to Y_{\mathbb C}^2$ given by
\[
A(s)\Phi:=
\left(\begin{array}{cc}
	\triangle \Phi_1+ d_1\nabla\cdot(\Phi_1 \nabla u_s)+ \lambda_1 (r_1(x)-2a_{11}u_s-a_{12} v_s) \Phi_1-\lambda_1 a_{12} u_s\Phi_2 \\
	\triangle \Phi_2+ d_2\nabla\cdot(\Phi_2 \nabla v_s) -\lambda_2 a_{21} v_s \Phi_1+ \lambda_2 (r_2(x)-a_{21}u_s-2a_{22}v_s)\Phi_2
\end{array} \right),
\]
and 
\[
B(s)\Phi:=
\left(\begin{array}{cc}
	d_1 \nabla\cdot (u_s \nabla \Phi_1)\\
	d_2 \nabla\cdot (v_s \nabla \Phi_2)
\end{array} \right).
\]
Therefore, $\mu\in\mathbb C$ is referred as an eigenvalue associated with $A_{\tau}(s)$ if and only if there exists $\Psi=(\phi_s, \psi_s)^T \in X_{\mathbb C}^2\backslash\{(0, 0)\}$ such that 
\begin{equation}\label{3.15}
	\Lambda(\mu, s, \tau) \Psi= (A(s)+B(s)e^{-\mu\tau}-\mu I)\Psi=0.
\end{equation}

Firstly, to get the estimated results, we give the following assumption:
\begin{enumerate}
    	\item[$(\bf{H_1})$] $|d_1|<d_1^*:=\dfrac{1}{\max\limits_{s\in[0, \delta]}\|u_s\|_{\infty}}$, \quad
        $|d_2|<d_2^*:=\dfrac{1}{\max\limits_{s\in[0, \delta]}\|v_s\|_{\infty}}$.
\end{enumerate}

\begin{lemma}\label{lemma3.1}
    Under the assumption $\mathbf{(H_{0})}$ and $\mathbf{(H_{1})}$, if $(\mu_s, s, \tau_s, \phi_s, \psi_s) \in \mathbb C\times (0,\delta]\times \mathbb R \times X_{\mathbb C}^2 \backslash \{ (0, 0) \}$ is a solution of \eqref{3.15} with $\mathrm{Re} \mu>0$, then there exists constants $M_2>0$, $M_3>0$ depending on $d_1$ and $d_2$, such that     \begin{equation}\label{3.16}
        \|\nabla\phi_s\|^2_{Y_{\mathbb C}}\le M_2\|\phi_s \|_{Y_{\mathbb C}}(\|\phi_s\|_{Y_{\mathbb C}}+\|\psi_s\|_{Y_{\mathbb C}}), \quad
        \|\nabla\psi_s\|^2_{Y_{\mathbb C}}\le M_3\|\psi_s \|_{Y_{\mathbb C}}(\|\phi_s\|_{Y_{\mathbb C}}+\|\psi_s\|_{Y_{\mathbb C}}).
    \end{equation}
\end{lemma}
\begin{proof}
    As in the embedding theorem \cite{AF2003}, for any $s\in(0, \delta]$, we have $u_s, v_s \in C^{1+\gamma}(\Bar{\Omega})$ and $u_s, v_s$ are bounded in $C^{1+\gamma}(\Bar{\Omega})$, where $\gamma\in(0,\frac{1}{2})$. Since $(u_s, v_s)$ is a solution of \eqref{2.2}, it follows from the regularity theory of elliptic equations \cite{GT1977} that $u_s, v_s\in C^{2+\gamma_1}(\Bar{\Omega})$ with $0<\gamma_1<min\{\gamma, \frac{1}{2}\}$, and $u_s, v_s$ are bounded in $C^{2+\gamma_1}(\Bar{\Omega})$, that is, there exists a constant $M_0=M_0(\gamma_1)>0$ such that 
    \begin{equation}\label{18}
    	|u_s|_{2+\gamma_1}\le M_0, \quad |v_s|_{2+\gamma_1}\le M_0.
    \end{equation}
    Denote 
    \begin{align}\label{3.18}
    	& A_{11}^s=\lambda_{1s} (r_1(x)-2a_{11}u_s-a_{12}v_s), \quad 
    	A_{12}^s=-a_{12}\lambda_{1s} u_s, \nonumber\\
    	& A_{21}^s=-a_{21}\lambda_{2s} v_s, \quad
    	A_{22}^s=\lambda_{2s} (r_2(x)-a_{21}u_s-2a_{22}v_s).
    \end{align}
    Substituting  $(\mu_s, \tau_s, \phi_s, \psi_s)$ into \eqref{3.15} gives
    \begin{equation}\label{3.19}
	\begin{cases}
	\triangle \phi_s+d_1\nabla\cdot(\phi_s\nabla u_s)+d_1\nabla\cdot(u_s\nabla \phi_s)e^{-\mu_s\tau_s} +A_{11}^s\phi_s+A_{12}^s\psi_s=0, \\
	\triangle \psi_s+d_2\nabla\cdot(\psi_s\nabla v_s)+d_2\nabla\cdot(v_s\nabla \psi_s)e^{-\mu_s\tau_s} +A_{21}^s\phi_s+A_{22}^s\psi_s=0,
	\end{cases}
	\end{equation}
    Multiplying $\phi_s$ and $\psi_s$ with both sides of the first and second equations of \eqref{3.19}, and integrating over $\Omega$, respectively, we have 
    \begin{align}\label{3.21}
    	\|\nabla\phi_s\|^2_{Y_{\mathbb C}} =&
    	d_1 \mathrm{Re}\{\int_{\Omega}\phi_s \nabla\cdot(\phi_s\nabla u_s)\mathrm{d}x\}
    	+d_1 \mathrm{Re}\{e^{-\mu_s\tau_s}\}\int_{\Omega}u_s|\nabla \phi_s|^2\mathrm{d}x \nonumber\\
    	&+ A_{11}^s \|\phi_s\|^2_{Y_{\mathbb C}}
    	+ A_{12}^s \langle \phi_s, \psi_s \rangle
    	-\mathrm{Re}(\mu_s) \|\phi_s\|^2_{Y_{\mathbb C}}, \nonumber\\
    	\|\nabla\psi_s\|^2_{Y_{\mathbb C}} =&
    	d_2 \mathrm{Re}\{\int_{\Omega}\psi_s \nabla\cdot(\psi_s\nabla v_s)\mathrm{d}x\}
    	+d_2 \mathrm{Re}\{e^{-\mu_s\tau_s}\}\int_{\Omega}v_s|\nabla \psi_s|^2\mathrm{d}x \nonumber\\
    	&+ A_{21}^s\langle \phi_s, \psi_s \rangle
    	+ A_{12}^s \|\psi_s\|^2_{Y_{\mathbb C}}
    	-\mathrm{Re}(\mu_s) \|\psi_s\|^2_{Y_{\mathbb C}}.
    \end{align}
    Since
    \[
    \int_{\Omega}\phi_s\nabla\cdot(\phi_s\nabla u_s)\mathrm{d}x
    =-\int_{\Omega}\phi_s\nabla\cdot(\phi_s\nabla u_s)\mathrm{d}x+\int_{\Omega}|\phi_s|^2\triangle u_s\mathrm{d}x,
    \]
    and
    \[
    \int_{\Omega}\psi_s\nabla\cdot(\psi_s\nabla v_s)\mathrm{d}x
    =-\int_{\Omega}\psi_s\nabla\cdot(\psi_s\nabla v_s)\mathrm{d}x+\int_{\Omega}|\psi_s|^2\triangle v_s\mathrm{d}x,
    \]
    we have 
    \begin{equation}\label{3.22}
        \mathrm{Re}\{\int_{\Omega}\phi_s\nabla\cdot(\phi_s\nabla u_s)\mathrm{d}x\}=\dfrac{1}{2}\int_{\Omega}|\phi_s|^2\triangle u_s\mathrm{d}x, \quad
        \mathrm{Re}\{\int_{\Omega}\psi_s\nabla\cdot(\psi_s\nabla v_s)\mathrm{d}x\}=\dfrac{1}{2}\int_{\Omega}|\psi_s|^2\triangle v_s\mathrm{d}x,
    \end{equation}
    Combining \eqref{18}, \eqref{3.21}, \eqref{3.22} and $\rm{Re} \mu_s>0$ yields
    \begin{align}\label{3.23}
    	\|\nabla\phi_s\|^2_{Y_{\mathbb C}} \le &
    	\dfrac{|d_1|M_0}{2} \|\phi_s \|^2_{Y_{\mathbb C}} 
    	+|d_1|\|u_s\|_{\infty} \|\nabla\phi_s\|^2_{Y_{\mathbb C}}
    	+ M_1(\|\phi_s\|^2_{Y_{\mathbb C}}+\|\phi_s\|_{Y_{\mathbb C}}\|\psi_s\|_{Y_{\mathbb C}}), \nonumber\\
    	\|\nabla\psi_s\|^2_{Y_{\mathbb C}} \le &
    	\dfrac{|d_2|M_0 }{2} \|\psi_s \|^2_{Y_{\mathbb C}} 
    	+|d_2|\|v_s\|_{\infty} \|\nabla\psi_s\|^2_{Y_{\mathbb C}}
    	+ M_1(\|\psi_s\|^2_{Y_{\mathbb C}}+\|\phi_s\|_{Y_{\mathbb C}}\|\psi_s\|_{Y_{\mathbb C}}).
    \end{align}
    where $M_1=\max\limits_{s\in[0, \delta]} \{ \|A_{11}^s\|_{\infty}, \|A_{12}^s\|_{\infty}, \|A_{21}^s\|_{\infty}, \|A_{22}^s\|_{\infty} \}$, which means that 
    \begin{align}\label{e3.24}
    	\|\nabla\phi_s\|^2_{Y_{\mathbb C}}\le 
        \dfrac{(\dfrac{|d_1|}{2}M_0+M_1) \|\phi_s \|^2_{Y_{\mathbb C}}
        	+ M_1\|\phi_s\|_{Y_{\mathbb C}}\|\psi_s\|_{Y_{\mathbb C}}}{1-|d_1|\|u_s\|_{\infty}}
        	=M_2\|\phi_s \|_{Y_{\mathbb C}}(\|\phi_s\|_{Y_{\mathbb C}}+\|\psi_s\|_{Y_{\mathbb C}}), \nonumber\\
        \|\nabla\psi_s\|^2_{Y_{\mathbb C}}\le 
        \dfrac{(\dfrac{|d_2|}{2}M_0+M_1) \|\psi_s \|^2_{Y_{\mathbb C}}
        	+ M_1\|\phi_s\|_{Y_{\mathbb C}}\|\psi_s\|_{Y_{\mathbb C}}}{1-|d_2|\|v_s\|_{\infty}}
        	=M_3\|\psi_s \|_{Y_{\mathbb C}}(\|\phi_s\|_{Y_{\mathbb C}}+\|\psi_s\|_{Y_{\mathbb C}}).
    \end{align}
    The proof is completed.
\end{proof}

Next, we will show the situation when the infinitesimal generator $A_{\tau}(s)$ has a pair of purely imaginary eigenvalues $\mu=\pm i\nu \ (\nu>0)$ for some $\tau\ge 0$ if and only if $\Lambda(\nu, s, \theta)\Psi=0$ is solvable for some $\nu>0$, i.e.
\begin{equation}\label{3.24}
	A(s)\Psi + B(s)\Psi e^{-i\theta}-i\nu \Psi=0,
\end{equation}
where $\theta:=\nu \tau \in [0, 2\pi)$ and $\Psi\in X_{\mathbb C}^2\backslash\{(0,0)\}$. For further discussion, we give the following two lemmas.

\begin{lemma}\label{lemma3.2}
	If $z\in X_{\mathbb C}$ and $\langle\phi, z\rangle=0$, then $\langle (\triangle+l r(x)) z, z\rangle\ge l_2 \langle z,z\rangle$, where $l_2$ is the second eigenvalue of the operator $-(\triangle+l r(x))$.
\end{lemma}
\begin{proof}
	Similar to the proof in \cite[Lemma 3.2]{BH1996}, we omit it here.
\end{proof}
\begin{lemma}\label{lemma3.3}
	Under the assumption $\mathbf{(H_{0})}$ and $\mathbf{(H_{1})}$, if $(\nu_s, s, \theta_s, \phi_s, \psi_s) \in \mathbb R^+\times (0,\delta]\times [0, 2\pi) \times X_{\mathbb C}^2 \backslash \{ (0, 0) \}$ is a solution of \eqref{3.24}, then $\dfrac{\nu_s}{s}$ is bounded for $s\in(0, \delta]$.
\end{lemma}
\begin{proof}
	Calculating the inner product of \eqref{3.24} with $\Psi=(\phi_s, \psi_s)^T$, we have
	\begin{equation}\label{3.25}
		\langle A(s)\Psi +B(s)\Psi e^{-i\theta_s}-i\nu_s \Psi, \Psi \rangle=0.
	\end{equation}
	Separating the real and imaginary parts of \eqref{3.25}, we obtain 
	\begin{align*}
	\nu_s \langle \Psi, \Psi \rangle = -\sin{\theta_s} \langle B(s)\Psi, \Psi \rangle
	= \sin{\theta_s}(d_1\langle u_s\nabla\phi_s, \nabla\phi_s \rangle
	+d_2\langle v_s\nabla\psi_s, \nabla\psi_s \rangle).
	\end{align*}
	Therefore,
	\begin{align*}
		\dfrac{| \nu_s |}{s} &= 
		\dfrac{1}{\| \phi_s \|_{Y_{\mathbb C}}^2+\| \psi_s \|_{Y_{\mathbb C}}^2}\sin{\theta_s}(|d_1|\|\cos{\omega}\phi_*+w_1(s)\|_{\infty}\|\nabla\phi_s\|^2_{Y_{\mathbb C}}+|d_2|\|\sin{\omega}\psi_*+w_2(s)\|_{\infty}\|\nabla\psi_s\|^2_{Y_{\mathbb C}}) \\
		&\le \dfrac{1}{\| \phi_s \|_{Y_{\mathbb C}}^2+\| \psi_s \|_{Y_{\mathbb C}}^2}(|d_1|M_4\|\nabla\phi_s\|^2_{Y_{\mathbb C}}+|d_2|M_5\|\nabla\psi_s\|^2_{Y_{\mathbb C}}),
	\end{align*}
	where $M_4=\max\limits_{s\in[0, \delta]} \{ \|\cos{\omega}\phi_*+w_1(s)\|_{\infty} \}$ and $M_5=\max\limits_{s\in[0, \delta]} \{ \|\sin{\omega}\psi_*+w_2(s)\|_{\infty} \}$. From the boundedness of $\|\nabla\phi_s\|^2_{Y_{\mathbb C}}$ and $\|\nabla\psi_s\|^2_{Y_{\mathbb C}}$ in \eqref{3.23}, we can obtain that $\dfrac{\nu_s}{s}$ is bounded for $s\in(0,\delta]$. The proof is completed. 
\end{proof}

In the following, we consider whether $A_{\tau}(s)$ has purely imaginary eigenvalues $\mu=\pm i\nu$ for some $\tau\ge 0$. Note that the decomposition $X_{\mathbb C}^2=\mathcal{\mathcal L_{\lambda_*}}\oplus (X_1)_{\mathbb C}^2$. Suppose that $(\nu, s, \theta, \varphi)$ is a solution of \eqref{3.24} with $\varphi=(\varphi_1, \varphi_2)^T\in X_{\mathbb C}^2\backslash\{(0,0)\}$, then we can rewrite $\varphi=(\varphi_1, \varphi_2)$ as 
\begin{align}\label{3.26}
	& \nu=sh, \quad h>0, \nonumber \\
    & \varphi=\Phi_*+(p_1+ip_2)\Psi_*+sz, \nonumber \\
    & \|\varphi\|^2_{Y_{\mathbb C}}=\|\Phi\|^2_{Y_{\mathbb C}}+\|\Psi\|^2_{Y_{\mathbb C}},
\end{align}
where $p_1$, $p_2\in \mathbb R$ and $z=(z_1, z_2)^T\in (X_1)_{\mathbb C}^2$. Substituting \eqref{3.26} and \eqref{2.6} into \eqref{3.24}, we have
\begin{equation}\label{3.27}
    \begin{cases}
        G_1(z_1, z_2, p_1, p_2, h, \theta, s):=
        &(\triangle+\lambda_1^*r_1(x))z_1 \\
        &+d_1\nabla\cdot[(\phi_*+sz_1)\nabla(\cos{\omega}\phi_*+w_1(s))] \\
        &+d_1\nabla\cdot[(\cos{\omega}\phi_*+w_1(s))\nabla (\phi_*+sz_1)]e^{-i\theta} \\
        &+(\phi_*+sz_1)[(\lambda_1'(0)+o(s))r_1(x)-ih] \\
        &-\lambda_{1s}(\phi_*+sz_1)[2a_{11}(\cos{\omega}\phi_*+w_1(s))+a_{12}(\sin{\omega}\psi_*+w_2(s))]\\ 
        &-((p_1+ip_2)\psi_*+sz_2)a_{12}\lambda_{1s}(\cos{\omega}\phi_*+w_1(s)) \\
        G_2(z_1, z_2, p_1, p_2, h, \theta, s):=
        &(\triangle+\lambda_2^*r_2(x))z_2 \\
        &+d_2\nabla\cdot[((p_1+ip_2)\psi_*+sz_2)\nabla(\sin{\omega}\psi_*+w_2(s))] \\
        &+d_2\nabla\cdot[(\sin{\omega}\psi_*+w_2(s))\nabla ((p_1+ip_2)\psi_*+sz_2)]e^{-i\theta} \\
        &-(\phi_*+sz_1)a_{21}\lambda_{2s}(\sin{\omega}\psi_*+w_2(s)) \\
        &+\lambda_{2s}((p_1+ip_2)\psi_*+sz_2)[(\lambda_2'(0)+o(s))r_2(x)-ih] \\
        &-((p_1+ip_2)\psi_*+sz_2)[a_{21}(\cos{\omega}\phi_*+w_1(s))+2a_{22}(\sin{\omega}\psi_*+w_2(s))]\\   
        G_3(z_1, z_2, p_1, p_2, h, \theta, s):=
        &(p_1^2+p_2^2-1)\|\psi_*\|^2_{Y_{\mathbb C}}+s^2\|z\|^2_{Y_{\mathbb C}}.
    \end{cases}
\end{equation}

We define the continuous differential map $G: (X_1)_{\mathbb C}^2\times \mathbb R^4\times [0,\delta]\to Y_{\mathbb C}^2\times \mathbb R$ with
\begin{equation}
	G(z_1, z_2, p_1, p_2, h, \theta, s):=(G_1, G_2, G_3).
\end{equation} 
Denote
\begin{align}\label{K}
	& K_1:=a_{11}\lambda_1^*\kappa_3=a_{11}\lambda_1^*\int_{\Omega}\cos{\omega}\phi_*^3\mathrm{d}x>0, \quad
	K_2:=a_{22}\lambda_2^*\kappa_6=a_{22}\lambda_2^*\int_{\Omega}\sin{\omega}\psi_*^3\mathrm{d}x>0, \nonumber\\
	& K_3:=a_{21}\lambda_2^*\kappa_8=a_{21}\lambda_2^*\int_{\Omega}\sin{\omega}\phi_*\psi_*^2\mathrm{d}x>0, \quad
	K_4:=a_{12}\lambda_1^*\kappa_7=a_{12}\lambda_1^*\int_{\Omega}\cos{\omega}\phi_*^2\psi_*\mathrm{d}x>0,
\end{align}
three domains (shown in Figure \ref{f1.1}):
\begin{align}
	D_1:= & \{ (d_1, d_2)\in \mathbb R^2| \dfrac{K_3\kappa_1}{K_4\kappa_2}d_1+\dfrac{|K_1K_3-K_2K_4|}{K_4\kappa_2}< d_2 <\dfrac{K_3\kappa_1}{K_4\kappa_2}d_1-\dfrac{|K_1K_3-K_2K_4|}{K_4\kappa_2} \}, \nonumber\\
	D_2:= & \{ (d_1, d_2)\in \mathbb R^2|-\dfrac{\kappa_1}{\kappa_2}d_1+\dfrac{K_1+K_2}{\kappa_2}< d_2 <-\dfrac{\kappa_1}{\kappa_2}d_1-\dfrac{K_1+K_2}{\kappa_2} \}, \nonumber\\
	D_3= & D_3^1\cup D_3^2:= \complement_{\mathbb R^2}(D_1\cup D_2),
\end{align}
where $\complement_{\mathbb R^2}D$ represent the complement of set $D$ in $\mathbb R^2$. Focusing on $(d_1, d_2)\in D_3$, we make the following two hypotheses:
\begin{enumerate}
    \item[$(\bf{H_2})$] $K_1K_3-K_2K_4<0$, $K_1-K_4>0$, $K_3-K_4>0$,
    \item[$(\bf{H_3})$] $K_1K_3-K_2K_4>0$, $K_1-K_4>0$, $K_3-K_4<0$,
\end{enumerate}
and define a critical value
\[
d^*:=\dfrac{K_1K_3-K_2K_4}{d_1\kappa_1K_3-d_2\kappa_2K_4}.
\]

\begin{figure}[htbp]
\centering
\begin{tikzpicture}
    \begin{axis}[
        axis lines=middle,
        xmin=-2, xmax=2,
        ymin=-2, ymax=2,
        width=8cm, height=8cm,
        grid=none,
        xtick=\empty, 
        ytick=\empty,
        xlabel={$d_1$}, 
        ylabel={$d_2$},
        domain=-3:3, 
        enlargelimits=true, 
        samples=100 
    ]  
    \addplot [
            pattern=north east lines, 
            pattern color=olive, 
            opacity=0.5, 
        ]
        coordinates {(0.25,0.75) (2.5,3) (-2,3)} ;
    \addplot [
            pattern=north east lines, 
            pattern color=olive, 
            opacity=0.5, 
        ]
        coordinates {(-0.25,-0.75) (-2.5,-3) (2,-3)} ;
    \addplot [
            pattern=north west lines, 
            pattern color=olive, 
            opacity=0.5, 
        ]
        coordinates {(0.75,0.25) (3,2.5) (3,-2)} ;
    \addplot [
            pattern=north west lines, 
            pattern color=olive, 
            opacity=0.5, 
        ]
        coordinates {(-0.75,-0.25) (-3,-2.5) (-3,2)} ;
    \addplot [
    		pattern=north east lines, 
            pattern color=c1, 
            opacity=0.5, 
        ]
        coordinates {(3,-2) (3,-3) (2,-3) (-3,2) (-3,3) (-2,3)} ;   
        \addplot[gray, thick] {-1*x+1};
        \addplot[gray, thick] {-1*x-1};
        \addplot[gray, thick] {x+0.5};
        \addplot[gray, thick] {x-0.5};
        \node at (axis cs:1,1) [fill=none, text=black] {$D_1$};
        \node at (axis cs:-1, 1) [fill=none, text=black] {$D_2$};
        \node at (axis cs:2, 0.5) [fill=none, text=black] {$D_3^2$};
        \node at (axis cs:1, 2) [fill=none, text=black] {$D_3^1$};
        \node at (axis cs:-0.5, -2) [fill=none, text=black] {$D_3^1$};
        \node at (axis cs:-2, -0.5) [fill=none, text=black] {$D_3^2$};
    \end{axis}
    \node at (5, 1) [fill=none, text=black] {$l_1$};
    \node at (6, 1) [fill=none, text=black] {$l_2$};
    \node at (6, 5) [fill=none, text=black] {$l_3$};
	\node at (6, 6) [fill=none, text=black] {$l_4$};
\end{tikzpicture}
\caption{The regions of $(d_1, d_2)$ in $D_1$, $D_2$ and $D_3$, where $l_1: d_2=-\frac{\kappa_1}{\kappa_2}d_1+\frac{K_1+K_2}{\kappa_2}$, $l_2: d_2=-\frac{\kappa_1}{\kappa_2}d_1-\frac{K_1+K_2}{\kappa_2}$, $l_3: d_2=\frac{K_3\kappa_1}{K_4\kappa_2}d_1+\frac{|K_1K_3-K_2K_4|}{K_4\kappa_2}$, $l_4: d_2=\frac{K_3\kappa_1}{K_4\kappa_2}d_1-\frac{|K_1K_3-K_2K_4|}{K_4\kappa_2}$.}
\label{f1.1}
\end{figure}
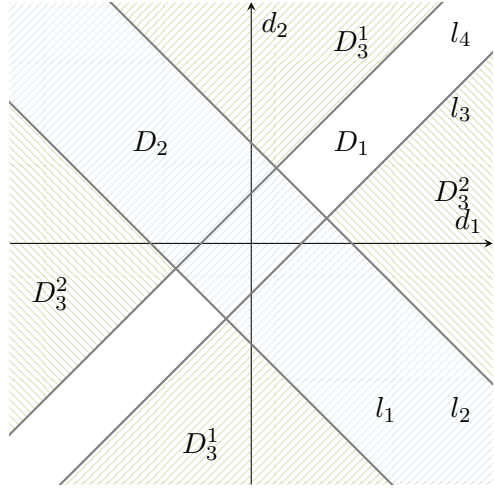

\begin{theorem}\label{th3.4}
	Under the assumption $\mathbf{(H_{0})}$, $\mathbf{(H_{1})}$ and $\mathbf{(H_{2})}$, if
	\begin{enumerate}
		\item[$\mathbf{(H_5)}$] $\dfrac{K_2\kappa_1}{K_1\kappa_2}<\dfrac{d_2}{d_1}<\dfrac{(K_2-K_3)\kappa_1}{(K_1-K_4)\kappa_2}$, and $(d_1, d_2)\in D_3^1$,
	\end{enumerate} 
    then $d^*\in[-1,1]$, there exists a unique solution of $G(z_1, z_2, p_1, p_2, h, \theta, s)=0$ when $s=0$, which can be expressed as
    	\begin{equation}\label{3.31}
    		W^0:=(z_1^0, z_2^0, p_1^0, p_2^0, h^0, \theta^0),
    	\end{equation}
    Here,
    \begin{align}\label{3.32}
		& p_1^0=\dfrac{d_2\kappa_2K_1-d_1\kappa_1K_2}{d_1\kappa_1K_3-d_2\kappa_2K_4}<0, \quad 
		p_2^0=\dfrac{\sqrt{(d_1\kappa_1K_3-d_2\kappa_2K_4)^2-(d_2\kappa_2K_1-d_1\kappa_1K_2)^2}}{d_2\kappa_2K_4-d_1\kappa_1K_3}<0, \nonumber\\
		& h^0=\dfrac{d_2\kappa_2K_4+d_1\kappa_1K_3}{d_1\kappa_1-d_2\kappa_2}p_2>0, \quad
		\theta^0=\arccos{d_*}>0,
	\end{align}
	where  $\kappa_1$, $\kappa_2$ and $K_i$, $(i=1,2,3,4)$, are defined in \eqref{kappa} and \eqref{K}, and then $(z_1^0, z_2^0)\in (X_1)_{\mathbb C}^2$ is the unique solution of 
	\begin{equation}\label{3.33}
    \begin{cases}
        0=& (\triangle+\lambda_1^*r_1(x))z_1+ d_1[\nabla\cdot(\phi_*\nabla(\cos{\omega}\phi_*))+\nabla\cdot(\cos{\omega}\phi_*\nabla \phi_*)e^{-i\theta^0}] \\
        &+\phi_*(\lambda_1'(0)r_1(x)-ih^0)-\lambda_1^*(2a_{11}\cos{\omega}\phi_*^2+a_{12}\sin{\omega}\phi_*\psi_*)-(p_1^0+ip_2^0)\lambda_1^*a_{12}\cos{\omega}\phi_*\psi_* , \\
        0=& (\triangle+\lambda_2^*r_2(x))z_2+ d_2(p_1^0+ip_2^0)[\nabla\cdot(\psi_*\nabla(\sin{\omega}\psi_*))+\nabla\cdot(\sin{\omega}\psi_*\nabla \psi_*)e^{-i\theta^0}] \\
        &+(p_1^0+ip_2^0)\psi_*(\lambda_2'(0)r_2(x)-ih^0)-a_{21}\lambda_2^*\sin{\omega}\phi_*\psi_*-\lambda_2^*(p_1^0+ip_2^0)(a_{21}\cos{\omega}\phi_*\psi_*+2a_{22}\sin{\omega}\psi_*^2).
    \end{cases}
	\end{equation}
\end{theorem}
\begin{proof}
    Firstly, when $s=0$, the equation $G(z_1, z_2, p_1, p_2, h, \theta, 0)=0$ is equivalent to 
    \begin{equation}\label{3.34}
    \begin{cases}
        0=& (\triangle+\lambda_1^*r_1(x))z_1+ d_1[\nabla\cdot(\phi_*\nabla(\cos{\omega}\phi_*))+\nabla\cdot(\cos{\omega}\phi_*\nabla \phi_*)e^{-i\theta}] \\
        &+\phi_*(\lambda_1'(0)r_1(x)-ih)-\lambda_1^*(2a_{11}\cos{\omega}\phi_*^2+a_{12}\sin{\omega}\phi_*\psi_*)-(p_1+ip_2)\lambda_1^*a_{12}\cos{\omega}\phi_*\psi_* , \\
        0=& (\triangle+\lambda_2^*r_2(x))z_2+ d_2(p_1+ip_2)[\nabla\cdot(\psi_*\nabla(\sin{\omega}\psi_*))+\nabla\cdot(\sin{\omega}\psi_*\nabla \psi_*)e^{-i\theta}] \\
        &+(p_1+ip_2)\psi_*(\lambda_2'(0)r_2(x)-ih)-a_{21}\lambda_2^*\sin{\omega}\phi_*\psi_*-\lambda_2^*(p_1+ip_2)(a_{21}\cos{\omega}\phi_*\psi_*+2a_{22}\sin{\omega}\psi_*^2) , \\
        0=& p_1^2+p_2^2-1.
    \end{cases}
	\end{equation}
	From the third equation of \eqref{3.34}, we have
	\begin{equation}\label{3.35}
		p_1^2+p_2^2=1.
	\end{equation}
	Multiplying the first and second equation of \eqref{3.34} with $\phi_*$ and $\psi_*$, respectively, and then integrating them over $\Omega$, gives that
	\begin{equation}\label{3.36}
    \begin{cases}
        d_1\kappa_1e^{-i\theta}-K_1-K_4(p_1+ip_2)=ih, \\
        d_2(p_1+ip_2)\kappa_2e^{-i\theta}-K_2(p_1+ip_2)-K_3 =ih(p_1+ip_2).
    \end{cases}
    \end{equation}
    Separating the real and imaginary parts of two equations for \eqref{3.36}, we obtain
    \begin{equation}\label{3.37}
    \begin{cases}
        d_1\kappa_1\cos{\theta}-K_1-K_4p_1=0, \\
        d_1\kappa_1\sin{\theta}+K_4p_2=-h, \\
        d_2\kappa_2\cos{\theta}-K_2-K_3p_1=0, \\
    	d_2\kappa_2\sin{\theta}-K_3p_2=-h.
    \end{cases}
    \end{equation}
	Combining assumption $(\bf{H_2})$, four equations of \eqref{3.37} and \eqref{3.35} gives \eqref{3.32}. The proof is completed.
\end{proof}

When $s>0$, from the implicit function theorem, the following result regarding the eigenvalue problem from $s$ to 0 can be obtained.
\begin{theorem}\label{th3.5}
	under the assumption $\mathbf{(H_{0})}$, $\mathbf{(H_{1})}$, $\mathbf{(H_{2})}$ and $\mathbf{(H_5)}$, there exists a constant $\bar{\delta}\in(0,\delta)$ and a continuously differentiable map $W(s): [0,\bar{\delta}]\to (X_1)_{\mathbb C}^2\times \mathbb R^4$ with $W(s):=(z_1^s, z_2^s, p_1^s, p_2^s, h^s, \theta^s)$, such that 
    \[
    G(W(s), s)=0, \quad W(0)=W^0,
    \]
    where $W^0$ is defined by \eqref{3.31}. Moreover, if there exists another $\Bar{W}(0):=(\Bar{z}_1^s, \Bar{z}_2^s, \Bar{p}_1^s, \Bar{p}_2^s, \Bar{h}^s, \Bar{\theta}^s)$ such that $G(\Bar{W}(s), s)=0$ with $\Bar{h}^s>0$ and $\Bar{\theta}^s>0$, then $\Bar{W}(s)=W(s)$.
\end{theorem}

\begin{proof}
    Define $T=(T_1, T_2, T_3): (X_1)_{\mathbb C}^2\times \mathbb R^4\to Y_{\mathbb C}^2\times \mathbb R$, which is the Fr\'{e}chet derivative of $G(z_1, z_2, p_1, p_2, h, \theta, s)$ with respect to $(z_1, z_2, p_1, p_2, h, \theta)$ evaluated at $W^0$:
    \[
    T:=D_{(z_1, z_2, p_1, p_2, h, \theta)}G(W^0, 0).
    \]
    A direct calculation gives    
    \begin{equation}\label{3.38}
    \begin{cases}
        T_1(\chi_1, \chi_2, \varrho_1, \varrho_2, \epsilon, \vartheta)
        =& (\triangle+\lambda_1^*r_1(x))\chi_1 \\
        &+(\varrho_1+i\varrho_2)a_{12}\lambda_1^*\cos{\omega}\phi_*\psi_* 
        - i\epsilon\phi_* 
        - id_1\cos{\omega}\nabla\cdot(\phi_*\nabla\phi_*)e^{-i\theta^0}\vartheta, \\
        T_2(\chi_1, \chi_2, \varrho_1, \varrho_2, \epsilon, \vartheta)
        =& (\triangle+\lambda_2^*r_2(x))\chi_2+(\varrho_1+i\varrho_2)d_2\sin{\omega}\nabla\cdot(\psi_*\nabla\psi_*(1+e^{-i\theta^0}) \\
        &+ (\varrho_1+i\varrho_2)\psi_*[\lambda_2'(0)r_2(x)-\lambda_2^*(a_{21}\cos{\omega}\phi_*+2a_{22}\sin{\omega}\psi_*)-ih^0] \\
        &- (p_1^0+ip_2^0)i\epsilon\psi_*
        -id_2(p_1^0+ip_2^0)\sin{\omega}\nabla\cdot(\psi_*\nabla\psi_*)e^{-i\theta^0}\vartheta, \\
        T_3(\chi_1, \chi_2, \varrho_1, \varrho_2, \epsilon, \vartheta)
        =& 2(p_1^0\varrho_1+ip_2^0\varrho_2)\|\psi\|_{Y_{\mathbb C}}^2.
    \end{cases}
	\end{equation}
	One can check that $T$ is bijective from $(X_1)_{\mathbb C}^2\times \mathbb R^4$ to $Y_{\mathbb C}^2\times \mathbb R$. Firstly, $L$ is an injection. If $L(\chi_1, \chi_2, \varrho_1, \varrho_2, \epsilon, \vartheta)=0$, it follows from the third equation of \eqref{3.38} that $\varrho_1=\varrho_2=0$. Substituting $\varrho_1=\varrho_2=0$ into \eqref{3.38}, we have
	\begin{equation}\label{3.39}
	\begin{cases}
		(\triangle+\lambda_1^*r_1(x))\chi_1-i\epsilon \phi_*-id_1\cos{\omega}\nabla\cdot(\phi_*\nabla\phi_*)e^{-i\theta^0}\vartheta=0, \\
		(\triangle+\lambda_2^*r_2(x))\chi_2-(p_1^0+ip_2^0)i\epsilon\psi_*-id_2(p_1^0+ip_2^0)\sin{\omega}\nabla\cdot(\psi_*\nabla\psi_*)e^{-i\theta^0}\vartheta =0.
	\end{cases}
	\end{equation}
	Taking the inner products of the first equation of \eqref{3.39} with $\phi_*$ gives that
	\begin{equation}\label{3.40}
		i\epsilon +id_1\kappa_1\cos{\omega}e^{-i\theta^0}\vartheta=0,
	\end{equation}
	separating the real and imaginary parts yields
	\begin{equation}\label{3.41}
	\begin{cases}
		\epsilon+d_1\kappa_1\cos{\omega}\cos{\theta^0}\vartheta=0, \\
		d_1\kappa_1\cos{\omega}\sin{\theta^0}\vartheta=0,
	\end{cases}
	\end{equation}
	which implies that $\vartheta=0$ and $\epsilon=0$. Since 
	$\varrho_1=\varrho_2=0$, $\vartheta=0$ and $\epsilon=0$, it can be inferred that $\chi_1=\chi_2=0$. Hence, $T$ is injective.
	
	Then, we show that $T$ is surjective. Suppose that for any $(\xi, \rho)\in Y_{\mathbb C}^2\times \mathbb C$, where $\rho=\rho_R+i\rho_I \in\mathbb C$, $\xi=(\xi_1, \xi_2)^T\in Y_{\mathbb C}^2$, let $\xi_1=\xi_1^1+\xi_1^2$ and $\xi_2=\xi_2^1+\xi_2^2$ for $(\xi_1^1, \xi_2^1)\in \mathcal N(\mathcal L_{\lambda_*})$ and $(\xi_1^2, \xi_2^2)\in (Y_1)_{\mathbb C}$ from the space decomposition $Y_{\mathbb C}^2=\mathcal N(\mathcal L_{\lambda_*})\oplus (Y_1)_{\mathbb C}$, there is 
	\begin{equation}\label{3.42}
		T(\chi_1, \chi_2, \varrho_1, \varrho_2, \epsilon, \vartheta)=(\xi, \rho)
	\end{equation}
	It follows from the third equation of \eqref{3.42} that
	\begin{equation}\label{3.43}
		\varrho_1=\dfrac{\rho_R}{2p_1^0}, \quad
		\varrho_2=\dfrac{\rho_I}{2p_2^0}.
	\end{equation}
	Calculating the inner products of the first and the second equation of \eqref{3.42} with $\phi_*$ and $\psi_*$, respectively, we have
	\begin{equation}\label{3.44}
		\begin{cases}
			(\varrho_1+i\varrho_2)K_4-i\epsilon-id_1\kappa_1e^{-i\theta^0}\vartheta=\langle \phi_*, \xi_1^1 \rangle, \\
			(\varrho_1+i\varrho_2)(d_2\kappa_2e^{-i\theta}-K_2-ih^0)-(p_1^0+ip_2^0)(i\epsilon+id_2\kappa_2e^{-i\theta^0}\vartheta)=\langle \psi_*, \xi_2^1 \rangle,
		\end{cases}
	\end{equation}
	separating the real and imaginary parts of the first equation of \eqref{3.44}, and then combining \eqref{3.43}, gives that
	\begin{equation}\label{3.45}
		\begin{cases}
			\dfrac{\rho_R}{2p_1^0}K_4-d_1\kappa_1\sin{\theta^0}\vartheta=\langle \phi_*, \xi_1^1 \rangle, \\
			\dfrac{\rho_I}{2p_2^0}K_4-\epsilon- d_1\kappa_1\cos{\theta^0}\vartheta=0.
		\end{cases}
	\end{equation}
	Then, $\vartheta$ and $\epsilon$ can be solved from the \eqref{3.45}. Finally, we put all values $\varrho_1$, $\varrho_2$, $\epsilon$ and $\vartheta$ into $T_1=\xi_1$ and $T_2=\xi_2$, $\chi_1$ and $\chi_2$ can be uniquely solved. Thus, $T$ is a surjection.
	
	From the implicit function theorem, there exists a constant $\bar{\delta}\in(0,\delta)$ and a continuous differentiable mapping $W(s): [0,\bar{\delta}]\to (X_1)_{\mathbb C}^2\times \mathbb R^4$, such that $G(W(s), s)=0$ with $W(0)=W^0$, defined in \eqref{3.31}. 
	
	Next, we prove the uniqueness of the solution $W(s)$. Suppose that there exists another mapping $\Bar{W}(s): [0,\Bar{\delta}]\to (X_1)_{\mathbb C}^2\times \mathbb R^4$ with $\Bar{W}(s)=(\Bar{z}_1^s, \Bar{z}_2^s, \Bar{p}_1^s, \Bar{p}_2^s, \Bar{h}^s, \Bar{\theta}^s)$ such that $G(\Bar{W}(s), s)=0$, then $\Bar{W}(s)\to W^0$ as $s\to 0$. We see that $\{ \Bar{h}^s \}$ is bounded as $\Bar{h}^s=\dfrac{\Bar{\nu}^s}{s}$ from Lemma \ref{lemma3.3}. And $\{ \Bar{\theta}^s \}$, $\{ \Bar{p}_1^s \}$, $\{ \Bar{p}_2^s \}$ are also bounded according to \eqref{3.27}. Then, we show the boundedness of $\Bar{z}_1^s$ and $\Bar{z}_2^s$. Calculating the inner products of the first and second equations of \eqref{3.27} with $\Bar{z}_1^s$ and $\Bar{z}_2^s$, respectively, we have 
	\begin{align}
		\|\nabla\Bar{z}_1^s\|_{Y_{\mathbb C}}^2 \le &
		\lambda_1r_1(x)\|\Bar{z}_1^s\|_{Y_{\mathbb C}}^2
		+|d_1|M_4s\|\nabla \Bar{z}_1^s\|_{Y_{\mathbb C}}^2e^{-i\theta}
		+|d_1|\|\phi_*\|_{Y_{\mathbb C}}\|\triangle (\cos{\omega}\phi_*+w_1(s))\|_{\infty}\|\Bar{z}_1^s\|_{Y_{\mathbb C}} \nonumber\\
		& + |d_1|(1+e^{-i\theta})\|\phi_*\|_{Y_{\mathbb C}}\|\nabla (\cos{\omega}\phi_*+w_1(s))\|_{\infty}\|\Bar{z}_1^s\|_{Y_{\mathbb C}}
		+|d_1|M_4\|\triangle\phi_*\|_{Y_{\mathbb C}}\|\Bar{z}_1^s\|_{Y_{\mathbb C}} \nonumber\\
		& + \dfrac{1}{2}|d_1|\|\triangle (\cos{\omega}\phi_*+w_1(s)) \|_{\infty}\|\Bar{z}_1^s\|_{Y_{\mathbb C}}^2
		+\lambda_1^*sM_1M_4(\|\Bar{z}_1^s\|_{Y_{\mathbb C}}^2+\|\Bar{z}_1^s\|_{Y_{\mathbb C}}\|\Bar{z}_2^s\|_{Y_{\mathbb C}}) \nonumber\\
		\le & |d_1|M_4s\|\nabla \Bar{z}_1^s\|_{Y_{\mathbb C}}^2
		+ M_6\|\Bar{z}_1^s\|_{Y_{\mathbb C}}^2 
		+ M_8\|\Bar{z}_1^s\|_{Y_{\mathbb C}} 
		+ \lambda_1^*sM_1M_4 \|\Bar{z}_1^s\|_{Y_{\mathbb C}}\|\Bar{z}_2^s\|_{Y_{\mathbb C}}, \nonumber\\
		\|\nabla\Bar{z}_2^s\|_{Y_{\mathbb C}}^2 \le &
		\lambda_2r_2(x)\|\Bar{z}_2^s\|_{Y_{\mathbb C}}^2
		+|d_2|M_5s\|\nabla \Bar{z}_2^s\|_{Y_{\mathbb C}}^2e^{-i\theta}
		+|d_2|\|\psi_*\|_{Y_{\mathbb C}}\|\triangle (\sin{\omega}\psi_*+w_2(s))\|_{\infty}\|\Bar{z}_2^s\|_{Y_{\mathbb C}} \nonumber\\
		& + |d_2|(1+e^{-i\theta})\|\psi_*\|_{Y_{\mathbb C}}\|\nabla (\sin{\omega}\psi_*+w_2(s))\|_{\infty}\|\Bar{z}_2^s\|_{Y_{\mathbb C}}
		+|d_2|M_5\|\triangle\psi_*\|_{Y_{\mathbb C}}\|\Bar{z}_2^s\|_{Y_{\mathbb C}} \nonumber\\
		& + \dfrac{1}{2}|d_2|\|\triangle (\sin{\omega}\psi_*+w_2(s)) \|_{\infty}\|\Bar{z}_2^s\|_{Y_{\mathbb C}}^2
		+\lambda_2^*sM_1M_5(\|\Bar{z}_2^s\|_{Y_{\mathbb C}}^2+\|\Bar{z}_1^s\|_{Y_{\mathbb C}}\|\Bar{z}_2^s\|_{Y_{\mathbb C}}) \nonumber\\
		\le & |d_2|M_5s\|\nabla \Bar{z}_2^s\|_{Y_{\mathbb C}}^2
		+ M_7 \|\Bar{z}_1^s\|_{Y_{\mathbb C}}^2
		+ M_9 \|\Bar{z}_2^s\|_{Y_{\mathbb C}} 
		+ \lambda_2^*sM_1M_5 \|\Bar{z}_1^s\|_{Y_{\mathbb C}}\|\Bar{z}_2^s\|_{Y_{\mathbb C}},
	\end{align}
	where $M_1$, $M_4$, $M_5$ are defined in Lemma \ref{lemma3.1} and Lemma \ref{lemma3.3}, and 
	\begin{align*}
		M_6 = & \|\lambda_1r_1(x)\|_{\infty}+\dfrac{1}{2}|d_1|\|\triangle (\cos{\omega}\phi_*+w_1(s)) \|_{\infty}+\lambda_1^*sM_1M_4, \\
		M_7 = & \|\lambda_2r_2(x)\|_{\infty}+\dfrac{1}{2}|d_2|\|\triangle (\sin{\omega}\psi_*+w_2(s)) \|_{\infty}+\lambda_2^*sM_1M_5, \\
		M_8 = & |d_1|\|\phi_*\|_{Y_{\mathbb C}}\|\triangle (\cos{\omega}\phi_*+w_1(s))\|_{\infty}+2|d_1|\|\phi_*\|_{Y_{\mathbb C}}\|\nabla (\cos{\omega}\phi_*+w_1(s))\|_{\infty} + |d_1|M_4\|\triangle\phi_*\|_{Y_{\mathbb C}}, \\
		M_9 = & |d_2|\|\psi_*\|_{Y_{\mathbb C}}\|\triangle (\sin{\omega}\psi_*+w_2(s))\|_{\infty}+2|d_2|\|\psi_*\|_{Y_{\mathbb C}}\|\nabla (\sin{\omega}\psi_*+w_2(s))\|_{\infty} + |d_2|M_5\|\triangle\psi_*\|_{Y_{\mathbb C}},
	\end{align*}
	Hence, 
	\begin{align*}
		\|\nabla\Bar{z}_1^s\|_{Y_{\mathbb C}}^2 \le &
		\dfrac{ M_6\|\Bar{z}_1^s\|_{Y_{\mathbb C}}^2 
		+ M_8\|\Bar{z}_1^s\|_{Y_{\mathbb C}} 
		+ \lambda_1^*sM_1M_4 \|\Bar{z}_1^s\|_{Y_{\mathbb C}}\|\Bar{z}_2^s\|_{Y_{\mathbb C}}}{1-|d_1|M_4s}, \\
		\|\nabla\Bar{z}_2^s\|_{Y_{\mathbb C}}^2 \le &
		\dfrac{ M_7\|\Bar{z}_2^s\|_{Y_{\mathbb C}}^2 
		+ M_9\|\Bar{z}_2^s\|_{Y_{\mathbb C}} 
		+ \lambda_2^*sM_1M_5 \|\Bar{z}_1^s\|_{Y_{\mathbb C}}\|\Bar{z}_2^s\|_{Y_{\mathbb C}}}{1-|d_2|M_5s}.
	\end{align*}
	Moreover, let $\lambda^{**}$ is the second eigenvalue of operator $-\mathcal L_{\lambda}$ it follows from Lemma \ref{lemma3.2} that
	\begin{align}
		\lambda_{**}\|\Bar{z}_1^s\|_{Y_{\mathbb C}}^2 \le \dfrac{M_6}{1-|d_1|M_4s}\|\Bar{z}_1^s\|_{Y_{\mathbb C}}^2
		+\dfrac{ M_8\|\Bar{z}_1^s\|_{Y_{\mathbb C}} 
		+ \lambda_1^*sM_1M_4 \|\Bar{z}_1^s\|_{Y_{\mathbb C}}\|\Bar{z}_2^s\|_{Y_{\mathbb C}}}{1-|d_1|M_4s}, \\
		\lambda_{**}\|\Bar{z}_2^s\|_{Y_{\mathbb C}}^2 \le \dfrac{M_7}{1-|d_2|M_5s}\|\Bar{z}_2^s\|_{Y_{\mathbb C}}^2
		+\dfrac{ M_9\|\Bar{z}_2^s\|_{Y_{\mathbb C}} 
		+ \lambda_1^*sM_1M_5 \|\Bar{z}_1^s\|_{Y_{\mathbb C}}\|\Bar{z}_2^s\|_{Y_{\mathbb C}}}{1-|d_2|M_5s}, 
	\end{align} 
	which implies that
	\[
	\|\Bar{z}_1^s\|_{Y_{\mathbb C}}^2 \le
	\dfrac{M_8\|\Bar{z}_1^s\|_{Y_{\mathbb C}} 
		+ \lambda_1^*sM_1M_4 \|\Bar{z}_1^s\|_{Y_{\mathbb C}}\|\Bar{z}_2^s\|_{Y_{\mathbb C}}}{\lambda_{**}(1-|d_1|M_4s)-M_6},
	\]
	and 
	\[
	\|\Bar{z}_2^s\|_{Y_{\mathbb C}}^2 \le
	\dfrac{M_9\|\Bar{z}_2^s\|_{Y_{\mathbb C}} 
		+ \lambda_2^*sM_1M_5 \|\Bar{z}_1^s\|_{Y_{\mathbb C}}\|\Bar{z}_2^s\|_{Y_{\mathbb C}}}{\lambda_{**}(1-|d_2|M_5s)-M_7}.
	\]
	Thus, $\{ \Bar{z}_1^s \}$ and $\{ \Bar{z}_2^s \}$ are bounded in $Y_{\mathbb C}$ for $s\in[0, \Bar{\delta}]$. Since $\mathcal L_{\lambda}$ has a bounded inverse $\mathcal L_{\lambda}^{-1}: (Y_1)_{\mathbb C}^2\to (X_1)_{\mathbb C}^2$, we find that $\{ \Bar{z}_1^s \}$ and $\{ \Bar{z}_2^s \}$ are also bounded in $X_{\mathbb C}$ by $\mathcal L_{\lambda}^{-1}G(\Bar{z}_1^s, \Bar{z}_2^s, \Bar{p}_1^s, \Bar{p}_2^s, \Bar{h}^s, \Bar{\theta}^s)=0$. Hence, $\{ \Bar{W}(s) \}$ is precompact in $Y_{\mathbb C}^2\times \mathbb R^4$ follows from the embedding theorem, which means there exists a subsequence $\{ \Bar{W}(s_n):= (\bar{z}_1^{s_n}, \bar{z}_2^{s_n}, \bar{p}_1^{s_n}, \bar{p}_2^{s_n}, \bar{h}^{s_n}, \bar{\theta}^{s_n}) \}$ such that 
	\[
	\lim\limits_{i\to\infty}s_n=0,
	\]
	and 
	\[
	\lim\limits_{i\to\infty} \Bar{W}(s_n)=\Bar{W}(0) \ \text{in} \  X_{\mathbb C}^2\times \mathbb R^4.
	\]
	Taking the limits of $\mathcal L_{\lambda}^{-1}(G_1(\Bar{W}(s_n)), G_2(\Bar{W}(s_n)))^T=(0, 0)^T$ as $i\to \infty$, we have $G(\Bar{W}(0), 0)=0$. Then, $\Bar{w}(0)=W^0$ since $G(z_1, z_2, p_1, p_2, h, \theta)=0$ has a unique solution $W^0$ defined in \eqref{3.31} by Thoerem \ref{th3.4}. The proof is completed. 
\end{proof}

Consequently, the following results could be obtained directly by Theorem \ref{thm7}:    

\begin{theorem}\label{thm7}
    Under the assumption $\mathbf{(H_{0})}$, $\mathbf{(H_{1})}$, $\mathbf{(H_{2})}$ and $\mathbf{(H_5)}$, for $s\in(0,\Bar{\delta})$, the eigenvalue problem
    \[
    \Lambda(\nu_n, s, \tau)\varphi=0, \quad \tau>0, \quad \varphi=(\varphi_1, \varphi_2)\in X_{\mathbb C}^2\backslash\{(0,0)\}, 
    \]
    has solution $(\nu, s, \varphi)$, if and only if 
    \begin{align}\label{3.49}
    	 &\nu=\nu^s:=sh^s, \quad \tau=\tau_n:=\dfrac{\theta^s+2n\pi}{\nu^s}, \ n=0, 1, 2, \dots \\
    	 &\varphi=c\varphi^s=c(\varphi_1^s, \varphi_2^s)^T=c\left(\begin{array}{cc}
            	\phi_*+sz_1^s \nonumber\\
            	(p_1^s+ip_2^s)\psi_*+sz_2^s
            	\end{array}\right),
    \end{align}
    where $c$ is a non-zero constant.
\end{theorem}

In the following, we show that the eigenvalues of \eqref{3.15} will pass through the imaginary axis from $i\nu_n$. The following estimate is given first.
\begin{lemma}\label{lemma3.7}
    Under the assumption $\mathbf{(H_{0})}$, $\mathbf{(H_{1})}$, $\mathbf{(H_{2})}$ and $\mathbf{(H_5)}$, we have
    \begin{equation}\label{3.50}
        S_n(s):=\int_{\Omega} (\varphi_1^s)^2+(\varphi_2^s)^2\mathrm{d}x
        +\tau_ne^{-i\theta^s}\left(d_1\int_{\Omega}\varphi_1^s\nabla\cdot(u_s\nabla\varphi_1^s)\mathrm{d}x
        +d_2\int_{\Omega}\varphi_2^s\nabla\cdot(v_s\nabla\varphi_2^s)\mathrm{d}x\right) \neq 0,
    \end{equation}
    for $s\in(0, \Bar{\delta})$ and $n=0, 1, 2, \dots$, where $u_s$, $v_s$ are defined as in \eqref{2.6}.
\end{lemma}
\begin{proof}
    Let $s\to 0$, we have
    \[
    \theta^s\to \theta^0, \quad (\varphi_1^s, \varphi_2^s)\to (\phi_*, (p_1^0+ip_2^0)\psi_*),
    \]
    and then
    \begin{align*}
    	S_n(0)=&\lim\limits_{s\to 0}S_n(s)= \int_{\Omega}\phi_*^2+(p_1^0+ip_2^0)^2\psi_*^2\mathrm{d}x
        +\dfrac{\theta^0+2n\pi}{h^0}(\cos{\theta^0-i\sin{\theta^0}})(d_1\kappa_1+d_2\kappa_2),
    \end{align*}
    separating the real and imaginary part of $S_n(0)$ gives that
	\begin{align}\label{3.51}
		\mathrm{Re}(S_n(0))=& \int_{\Omega}\phi_*^2+((p_1^0)^2-(p_2^0)^2)\psi_*^2\mathrm{d}x
		+\dfrac{\theta^0+2n\pi}{h^0}\left(\cos{\theta^0}(d_1\kappa_1+d_2\kappa_2)+2p_1^0p_2^0d_2\kappa_2\sin{\theta^0}\right), \nonumber\\
		\mathrm{Im}(S_n(0))=& 2p_1^0p_2^0\int_{\Omega}\psi_*^2\mathrm{d}x
		+\dfrac{\theta^0+2n\pi}{h^0}\left(2p_1^0p_2^0d_2\kappa_2\cos{\theta^0}-(d_1\kappa_1+d_2\kappa_2)\sin{\theta^0}\right).
     \end{align}
     Combining \eqref{3.51} and \eqref{3.37} yields
     \begin{align*}
     	\mathrm{Im}(S_n(0))=& 2p_1^0p_2^0\int_{\Omega}\psi_*^2\mathrm{d}x
     	+2(\theta^0+2n\pi)
		+\dfrac{\theta^0+2n\pi}{h^0}\left(2p_1^0p_2^0(K_2+K_3p_1^0)+(K_4-K_3)p_2^0\right) \\
		=& 2p_1^0p_2^0+2(\theta^0+2n\pi)
		+\dfrac{\theta^0+2n\pi}{h^0}(p_2^0(K_4-K_3)+2p_1^0p_2^0\dfrac{d_2\kappa_2(K_1K_3-K_2K_4)}{d_1\kappa_1K_3-d_2\kappa_2K_4}).
     \end{align*}
     Under the assumption $\mathbf{(H_{2})}$ (or $\mathbf{(H_{3})}$), it can be seen that
     \[
      p_1^0p_2^0>0, \quad  p_2^0(K_4-K_3)>0, \quad 
      \dfrac{d_2\kappa_2(K_1K_3-K_2K_4)}{d_1\kappa_1K_3-d_2\kappa_2K_4}>0,
      \]
      which means that $\mathrm{Im}(S_n(0))>0$. This proves $S_n(0)\neq 0$.
\end{proof}

Next, we show the transversality conditions to ensure the existence of periodic solutions.
\begin{theorem}\label{thm9}
	 Under the assumption $\mathbf{(H_{0})}$, $\mathbf{(H_{1})}$, $\mathbf{(H_{2})}$ and $\mathbf{(H_5)}$, let $\tau_n$ be defined as in \eqref{3.49}, for any $n=0, 1, 2, \dots$, we have 
	 \begin{enumerate}
	 	\item[(i)] $\mu=i\nu^s$ is a simple eigenvalue of $A_{\tau_n}(s)$;
	 	\item[(ii)] the transversality conditions $\mathrm{Re}\left(\dfrac{\mathrm{d}\mu}{\mathrm{d}\tau_n}(\tau_n)\right)> 0$ holds.
	 \end{enumerate}
\end{theorem}
\begin{proof}
	We first prove the part (i). From Theorem \ref{thm9}, we have
	\[
	\mathcal N[A_{\tau_n}-i\nu^sI]=\mathrm{Span}\{ e^{-\nu^st}\varphi^s \}, \quad t\in[-\tau_n, 0],
	\]
	where $\varphi^s$ is defined as in Theorem \ref{thm9}. Suppose that $\mu=i\nu^s$ is not a simple eigenvalue, which means that there exists $\tilde \varphi=(\tilde\varphi_1, \tilde\varphi_2)^T\in\mathcal N[A_{\tau_n}-i\nu^sI]^2$, i.e.,
	\[
	[A_{\tau_n}(s)-i\nu^sI]\tilde\varphi\in\mathcal N[A_{\tau_n}(s)-i\nu^sI]=\mathrm{Span}\{ e^{i\nu^st}\varphi^s \}.
	\]
	Hence, there exists a constant $\alpha$ such that $[A_{\tau_n}(s)-i\nu^sI]\tilde\varphi=\alpha e^{i\nu^st}\varphi^s$. And then there holds that
	\begin{align}\label{3.52}
		\tilde\varphi'(t)=& i\nu^s\tilde\varphi+\alpha e^{i\nu^st}\varphi^s, \quad t\in[-\tau_n, 0], \nonumber\\
		\tilde\varphi'(0)=& A(s)\tilde\varphi(0)+B(s)\tilde\varphi(-\tau_n).
	\end{align}
	It follows from the first equation of \eqref{3.52} that
	\begin{align}\label{3.53}
		\tilde\varphi(t)=& \tilde\varphi(0)e^{i\nu^st}+\alpha te^{i\nu^st}\varphi^s, \nonumber\\
		\tilde\varphi'(0)=& i\nu^s\tilde\varphi(0)+\alpha\varphi^s.
	\end{align}
	Substituting $t=-\tau_n$ into the first equation of \eqref{3.53}, and then combining the second equation of \eqref{3.52} gives that
	\begin{align}\label{3.54}
		\Lambda(i\nu^s, s, \tau_n)\tilde\varphi(0)=& [A(s)+B(s)e^{-i\theta^s}-i\nu^sI]\tilde\varphi(0) \nonumber\\
		=& \alpha(\varphi^s+\tau_ne^{-i\theta^s}B(s)\varphi^s),
	\end{align}
	Taking the inner product of \eqref{3.54} with $\bar\varphi^s$, we have
	\begin{align*}
		0=& \langle \Lambda(-i\nu^s, s, \tau_n)\bar\varphi^s, \tilde\varphi(0) \rangle
		=\langle \bar\varphi^s, \Lambda(i\nu^s, s, \tau_n)\tilde\varphi(0) \rangle \\
		=& \alpha\langle \bar\varphi^s, \varphi^s+\tau_ne^{-i\theta^s}B(s)\varphi^s \rangle
		:= \alpha S_n(s),
	\end{align*}
	where $S_n(s)\neq 0$ is defined as in Lemma \ref{lemma3.7}. That is $\alpha=0$. Hence, $\tilde\varphi\in\mathcal N[A_{\tau_n}-i\nu^sI]$. It can be derived by induction as
	\[
	\mathcal N[A_{\tau_n}(s)-i\nu^sI]^j=\mathcal N[A_{\tau_n}-i\nu^s], \quad j=2,3,\dots.
	\]
	and $\mu=i\nu^s$ is a simple eigenvalue of $A_{\tau_n}(s)$ for $n=0,1,2,\dots$ by contradiction. This completes the proof of part (i).
	
	Next, we show the transversality conditions in part (ii). It follows from the implicit function theorem that there exists a neighborhood $O_n\times D_n\times H_n^2\subset \mathbb R\times\mathbb C\times X_{\mathbb C}^2$ of $(\tau_n, i\nu^s, \varphi^s)$ and a continuous differential function $(\mu, \varphi): O_n\to D_n\times H_n^2$ such that $\mu(\tau)$ is the unique eigenvalue of $A_{\tau}(s)$ in $D_n$ with the associated eigenfunction $\varphi^s(\tau)$, that is,
    \begin{align}\label{3.55}
        &\Lambda(\mu(\tau), s, \tau)\varphi^s(\tau)=A(s)\varphi^s(\tau)+B(s)e^{-i\theta^s}\varphi^s(\tau)-i\nu^s\varphi^s(\tau) = 0,
    \end{align}
    with $\mu(\tau_n)=i\nu^s$ and $\varphi^s(\tau_n)=\varphi^s$. Differentiating \eqref{3.55} with respect to $\tau$ at $\tau=\tau_n$ yields
    \begin{equation}\label{3.56}
    	\dfrac{\mathrm{d}\mu(\tau_n)}{\mathrm{d}\tau}(\varphi^s+\tau_ne^{-i\theta^s}B(s)\varphi^s)
    	=\Lambda(i\nu^s, s, \tau_n)\dfrac{\mathrm{d}\varphi^s(\tau_n)}{\mathrm{d}\tau}-i\nu^se^{-i\theta^s}B(s)\varphi^s,
    \end{equation}
    taking the inner product of \eqref{3.56} with $\bar\varphi^s$ gives that
    \begin{align*}
    	\dfrac{\mathrm{d}\mu(\tau_n)}{\mathrm{d}\tau}
    	=& \dfrac{\langle \varphi^s, -i\nu^se^{-i\theta^s}B(s)\varphi^s \rangle}{\langle \varphi^s, \varphi^s+\tau_ne^{-i\theta^s}B(s)\varphi^s \rangle} \\
    	=& -\dfrac{1}{|S_n|^2} \left( \langle \varphi^s, \varphi^s \rangle\langle \varphi^s, i\nu^se^{-i\theta^s}B(s)\varphi^s \rangle
    	+ i\nu^s\tau_n|\langle \varphi^s, e^{-i\theta^s}B(s)\varphi^s \rangle|^2 \right) \\
    	=& -\dfrac{1}{|S_n|^2}i\nu^se^{-i\theta^s}\int_{\Omega} (\varphi_1^s)^2+(\varphi_2^s)^2\mathrm{d}x\left( d_1\int_{\Omega}\varphi_1^s\nabla\cdot(u_s\nabla\varphi_1^s)\mathrm{d}x+ d_2\int_{\Omega}\varphi_2^s\nabla\cdot(v_s\nabla\varphi_2^s)\mathrm{d}x \right) \\
    	&- \dfrac{1}{|S_n|^2} i\nu^se^{-i\theta^s}\tau_n \left(d_1\int_{\Omega}\varphi_1^s\nabla\cdot(u_s\nabla\varphi_1^s)\mathrm{d}x+ d_2\int_{\Omega}\varphi_2^s\nabla\cdot(v_s\nabla\varphi_2^s)\mathrm{d}x\right)^2.
    \end{align*}
    Let $s\to 0$, it follows from \eqref{3.37} and \eqref{3.50} that
    \begin{align*}
    	\lim\limits_{s\to 0} \dfrac{1}{s^2}\mathrm{Re}(\dfrac{\mathrm{d}\mu(\tau_n)}{\mathrm{d}\tau})
    	=& \dfrac{2h^0}{|S_n|^2}(d_1\kappa_1+d_2\kappa_2)\sin{\theta^0}= -\dfrac{2h^0}{|S_n|^2}(K_4-K_3)p_2^0
    	+\dfrac{4(h^0)^2}{|S_n|^2}>0.
    \end{align*}
    The proof is completed.
\end{proof}

Moreover, similar to the discussion under the assumption $\mathbf{(H_{2})}$, we obtain the following results for the assumption $\mathbf{(H_{3})}$.  
\begin{theorem}\label{th3.9}
	Under the assumption $\mathbf{(H_{0})}$, $\mathbf{(H_{1})}$ and $\mathbf{(H_{3})}$, if
	\begin{enumerate}
		\item[$\mathbf{(H_{6})}$] $\dfrac{K_2\kappa_1}{K_1\kappa_2}<\dfrac{d_2}{d_1}<\dfrac{(K_2+K_3)\kappa_1}{(K_1+K_4)\kappa_2}$ and $(d_1, d_2)\in D_3^2$,
	\end{enumerate}
	then, 
	\begin{enumerate}
		\item[(i)] there exists a constant $\tilde{\delta}\in(0,\delta)$ and a unique continuously differentiable map $\tilde W(s): [0,\tilde {\delta}]\to (X_1)_{\mathbb C}^2\times \mathbb R^4$ with $\tilde W(s):=(\tilde z_1^s, \tilde z_2^s, \tilde p_1^s, \tilde p_2^s, \tilde h^s, \tilde \theta^s)$, such that 
    	\[
    	G(\tilde W(s), s)=0, \quad \tilde W(0)=\tilde W^0,
    	\]
    	where $\tilde W^0:=(z_1^0, z_2^0, p_1^0, p_2^0, h^0, \theta^0)$ with 
    	\begin{align}\label{59}
			& p_1^0=\dfrac{d_2\kappa_2K_1-d_1\kappa_1K_2}{d_1\kappa_1K_3-d_2\kappa_2K_4}>0, \quad 
			p_2^0=\dfrac{\sqrt{(d_1\kappa_1K_3-d_2\kappa_2K_4)^2-(d_2\kappa_2K_1-d_1\kappa_1K_2)^2}}{d_2\kappa_2K_4-d_1\kappa_1K_3}>0, \nonumber\\
			& h^0=\dfrac{d_2\kappa_2K_4+d_1\kappa_1K_3}{d_1\kappa_1-d_2\kappa_2}p_2>0, \quad
			\theta^0=\arccos{d_*}\in[0, 2\pi].
		\end{align}
	
		\item[(ii)] for $s\in(0,\tilde {\delta})$, the eigenvalue problem
    	\[
    	\Lambda(\nu_n, s, \tau)\varphi=0, \quad \tau>0, \quad \varphi=(\varphi_1, \varphi_2)\in X_{\mathbb C}^2\backslash\{(0,0)\}, 
    	\]
    	has solution $(\nu, s, \tilde \varphi)$, if and only if 
    	\begin{align}\label{60}
    		&\nu=\tilde \nu^s:=s\tilde h^s, \quad \tilde \tau=\tilde \tau_n:=\dfrac{\tilde \theta^s+2n\pi}{\tilde \nu^s}, \ n=0, 1, 2, \dots \\
    		&\tilde \varphi=c\tilde \varphi^s=c(\tilde \varphi_1^s, \tilde \varphi_2^s)^T=c\left(\begin{array}{cc}
            	\phi_*+s\tilde z_1^s \nonumber\\
            	(\tilde p_1^s+i\tilde p_2^s)\psi_*+s\tilde z_2^s
            	\end{array}\right),
    	\end{align}
    	where $c$ is a non-zero constant.
    
    	\item[(iii)] for any $n=0, 1, 2, \dots$, $\mu=i\tilde \nu^s$ is a simple eigenvalue of $A_{\tilde \tau_n}(s)$, and the transversality conditions $\mathrm{Re}\left(\dfrac{\mathrm{d}\mu}{\mathrm{d}\tilde \tau_n}(\tilde \tau_n)\right)> 0$ holds.
	\end{enumerate}
\end{theorem}

Next, we show the results for the steady state solution $(u_s,v_s)$ to be stable. 
\begin{theorem}\label{th3.10}
	Under the assumption $\mathbf{(H_0)}$, $\mathbf{(H_1)}$ and $(d_1, d_2)\in D_2$, there exists a constant $\tilde\delta\in(0, \delta]$, such that all the eigenvalues of $A_{\tau}(s)$ have negative real parts for $\tau\ge 0$ and $s\in(0, \tilde\delta)$.
\end{theorem}
\begin{proof}
	Firstly, when $\tau=0$, the stability of $(u_s, v_s)$ is determined by:
	\begin{equation}\label{3.57}
		A_0(s)\varphi=A(s)\varphi+B(s)\varphi-\mu\varphi.
	\end{equation} 
	Suppose that there exists a sequence $\{ (s_i, \mu^{s_i}, \varphi^{s_i} \}\subset (0, \tilde\delta] \times \mathbb C \times X_{\mathbb C}^2\backslash\{(0, 0)\} $ such that $s_i\to 0$ as $i\to\infty$ and $\mathrm{Re}(\mu^{s_i})\ge 0$ for $i\ge 1$. By the decomposition $X_{\mathbb C}^2=\mathcal{\mathcal L_{\lambda_*}}\oplus (X_1)_{\mathbb C}^2$, the eigenfunction $\varphi^{s_i}=(\varphi_1^{s_i}, \varphi_2^{s_i})$ can be rewritten as
	\begin{align}\label{3.58}
   		& \varphi=\Phi_*+(p_1^{s_i}+ip_2^{s_i})\Psi_*+s_iz^{s_i}, \quad p_1^{s_i}>0 \nonumber \\
    	& \|\varphi^{s_i}\|^2_{Y_{\mathbb C}}=\|\Phi\|^2_{Y_{\mathbb C}}+\|\Psi\|^2_{Y_{\mathbb C}},
	\end{align}
	where $z^{s_i}=(z_1^{s_i}, z_2^{s_i})^T\in (X_1)_{\mathbb C}^2$. Substituting \eqref{3.58}, \eqref{2.6} and $\mu^{s_i}=s_ih^{s_i}$ into \eqref{3.57}, we have
	\begin{equation}\label{3.59}
    \begin{cases}
        G_4(z_1^{s_i}, z_2^{s_i}, p_1^{s_i}, p_2^{s_i}, h^{s_i}, s_i):=
        &(\triangle+\lambda_1^*r_1(x))z_1^{s_i} \\
        &+d_1\nabla\cdot[(\phi_*+s_iz_1^{s_i})\nabla(\cos{\omega}\phi_*+w_1(s))] \\
        &+d_1\nabla\cdot[(\cos{\omega}\phi_*+w_1(s))\nabla (\phi_*+s_iz_1^{s_i})] \\
        &+(\phi_*+s_iz_1^{s_i})[(\lambda_1'(0)+o(s))r_1(x)-h^{s_i}] \\
        &-\lambda_{1s}(\phi_*+s_iz_1^{s_i})[2a_{11}(\cos{\omega}\phi_*+w_1(s))+a_{12}(\sin{\omega}\psi_*+w_2(s))]\\ 
        &-((p_1^{s_i}+ip_2^{s_i})\psi_*+s_iz_2^{s_i})a_{12}\lambda_{1s}(\cos{\omega}\phi_*+w_1(s)) \\
        G_5(z_1^{s_i}, z_2^{s_i}, p_1^{s_i}, p_2^{s_i}, h^{s_i}, s_i):=
        &(\triangle+\lambda_2^*r_2(x))z_2^{s_i} \\      
        &+ d_2\nabla\cdot[((p_1^{s_i}+ip_2^{s_i}) \psi_*+s_iz_2^{s_i}) \nabla(\sin{\omega}\psi_*+w_2(s))] \\
        &+d_2\nabla\cdot[(\sin{\omega}\psi_*+w_2(s))\nabla ((p_1^{s_i}+ip_2^{s_i})\psi_*+s_iz_2^{s_i})] \\
        &-(\phi_*+s_iz_1^{s_i})a_{21}\lambda_{2s}(\sin{\omega}\psi_*+w_2(s)) \\
        &+\lambda_{2s}((p_1^{s_i}+ip_2^{s_i})\psi_*+s_iz_2^{s_i})[(\lambda_2'(0)+o(s))r_2(x)-h^{s_i}] \\
        &-((p_1^{s_i}+ip_2^{s_i})\psi_*+s_iz_2^{s_i})[a_{21}(\cos{\omega}\phi_*+w_1(s))\\
        &+2a_{22}(\sin{\omega}\psi_*+w_2(s))]\\   
        G_6(z_1^{s_i}, z_2^{s_i}, p_1^{s_i}, p_2^{s_i}, h^{s_i}, s_i):=
        &((p_1^{s_i})^2+(p_2^{s_i})^2-1)\|\psi_*\|^2_{Y_{\mathbb C}}+s^2\|z_1^{s_i}\|^2_{Y_{\mathbb C}}.
    \end{cases}
	\end{equation}
	and a continuously differentialable mapping from $(X_1)_{\mathbb C}^2\times\mathbb R^+\times \mathbb R^3$ to $Y_{\mathbb C}^2\times \mathbb R$ with 
	\[
	\tilde{G}(z_1^{s_i}, z_2^{s_i}, p_1^{s_i}, p_2^{s_i}, h^{s_i}, s_i)=(G_4, G_5, G_6)^T=0.
	\]
	Similar to the proof of Theorem \ref{thm7}, we obtain $\{z_1^{s_i}\}$, $\{z_2^{s_i}\}$ are bounded in $Y_{\mathbb C}$, $\{p_1^{s_i}\}$, $\{p_2^{s_i}\}$ and $\{h^{s_i}\}$ are bounded for $s\in(0,\tilde\delta)$. On the other hand, $\{z_1^{s_i}\}$ and $\{z_2^{s_i}\}$ are also bounded in $X_{\mathbb C}$ since the operator $\mathcal N(\lambda_*): (X_1)_{\mathbb C}^2\mapsto (Y_1)_{\mathbb C}^2$ has a bounded inverse. Hence, $\{ (z_1^{s_i}, z_2^{s_i}, p_1^{s_i}, p_2^{s_i}, h^{s_i})\}_{i=1}^{\infty}$ is precompact in $Y_{\mathbb C}^2\times \mathbb R^3$. That is, there exists a subsequence $\{ (z_1^{s_{i_j}}, z_2^{s_{i_j}}, p_1^{s_{i_j}}, p_2^{s_{i_j}}, h^{s_{i_j}})\}_{j=1}^{\infty}$ satisfies $s_{i_j}\to 0$ as $j\to\infty$, and
	\[
	(z_1^{s_{i_j}}, z_2^{s_{i_j}}, p_1^{s_{i_j}}, p_2^{s_{i_j}}, h^{s_{i_j}})\to (z_1^*, z_2^*, p_1^*, p_2^*, h^*)\in Y_{\mathbb C}^2\times\mathbb R^+ \times\mathbb R \times\mathbb C, 
	\]
	with $\mathrm{Re}(h^*)\ge 0$. Applying the inverse operator $\mathcal N(\lambda_*)^{-1}$ on the both sided of $\bar G=0$, i.e.,
	\begin{equation}\label{3.60}
		\mathcal N(\lambda_*)^{-1}\bar G(z_1^{s_{i_j}}, z_2^{s_{i_j}}, p_1^{s_{i_j}}, p_2^{s_{i_j}}, h^{s_{i_j}}, s_{i_j})=0,
	\end{equation} 
	and then let $j\to\infty$, it follows from \eqref{3.60} that $(z_1^*, z_2^*)\in (X_1)_{\mathbb C}^2$ satisfies
	\begin{equation}\label{3.61}
    \begin{cases}
        0=& (\triangle+\lambda_1^*r_1(x))z_1^*+ d_1[\nabla\cdot(\phi_*\nabla(\cos{\omega}\phi_*))+\nabla\cdot(\cos{\omega}\phi_*\nabla \phi_*)] \\
        &+\phi_*(\lambda_1'(0)r_1(x)-h^*)-\lambda_1^*(2a_{11}\cos{\omega}\phi_*^2+a_{12}\sin{\omega}\phi_*\psi_*)-(p_1^*+ip_2^*)\lambda_1^*a_{12}\cos{\omega}\phi_*\psi_* , \\
        0=& (\triangle+\lambda_2^*r_2(x))z_2+ d_2(p_1^*+ip_2^*)[\nabla\cdot(\psi_*\nabla(\sin{\omega}\psi_*))+\nabla\cdot(\sin{\omega}\psi_*\nabla \psi_*)] \\
        &+(p_1^*+ip_2^*)\psi_*(\lambda_2'(0)r_2(x)-h^*)-a_{21}\lambda_2^*\sin{\omega}\phi_*\psi_*-\lambda_2^*(p_1^*+ip_2^*)(a_{21}\cos{\omega}\phi_*\psi_*+2a_{22}\sin{\omega}\psi_*^2) , \\
        0=& (p_1^*)^2+(p_2^*)^2-1.
    \end{cases}
	\end{equation}
	Taking the inner products of the first and the second equations of \eqref{3.61} with $\phi_*$ and $\psi_*$, respectively, we obtain
	\begin{equation}\label{3.62}
    \begin{cases}
        d_1\kappa_1-K_1-K_4(p_1^*+ip_2^*)=h^*, \\
        d_2(p_1^*+ip_2^*)\kappa_2-K_2(p_1^*+ip_2^*)-K_3 =(p_1^*+ip_2^*)h^*.
    \end{cases}
    \end{equation}
    Separating the real and imaginary parts of two equations for \eqref{3.36}, we obtain
    \begin{equation}\label{3.63}
    \begin{cases}
        d_1\kappa_1-K_1-K_4p_1^*=\mathrm{Re}(h^*), \\
        -K_4p_2^*=\mathrm{Im}(h^*), \\
        d_2\kappa_2-K_2-K_3p_1^*=\mathrm{Re}(h^*), \\
    	K_3p_2^*=\mathrm{Im}(h^*).
    \end{cases}
    \end{equation}
    From the second, the fourth equations of \eqref{3.63} and the third equation of \eqref{3.61}, we see that $p_2^*=0$ and $p_1^*=1$. Therefore, It follows from the first and the third equations of \eqref{3.63} that 
    \[
    \mathrm{Re}(h^*)=\dfrac{(d_1\kappa_1+d_2\kappa_2)-(K_1+K_2)-(K_3+K_4)}{2}<0,
    \]
    which is a contradiction to $\mathrm{Re}(h^*)\ge 0$.
    
    Next, we consider the $\tau>0$ situation by applying a similar method to the case $\tau=0$. Suppose that there exists a sequence $\{ (s_n, \mu^{s_n}, \varphi^{s_n} \}\subset (0, \tilde\delta] \times \mathbb C \times X_{\mathbb C}^2\backslash\{(0, 0)\} $ such that $s_n\to 0$ as $n\to\infty$ and $\mathrm{Re}(\mu^{s_n})\ge 0$ for $i\ge 1$. We rewrite the eigenfunction $\varphi^{s_n}=(\varphi_1^{s_n}, \varphi_2^{s_n})$ as
	\begin{align}\label{3.64}
   		& \varphi=\Phi_*+(p_1^{s_n}+ip_2^{s_n})\Psi_*+s_nz^{s_n}, \quad p_1^{s_n}>0 \nonumber \\
    	& \|\varphi^{s_n}\|^2_{Y_{\mathbb C}}=\|\Phi\|^2_{Y_{\mathbb C}}+\|\Psi\|^2_{Y_{\mathbb C}},
	\end{align}
	where $z^{s_n}=(z_1^{s_n}, z_2^{s_n})^T\in (X_1)_{\mathbb C}^2$. Substituting \eqref{3.64}, \eqref{2.6} and $\mu^{s_n}=s_nh^{s_n}$ into \eqref{3.64}, we have
	\begin{equation}\label{3.65}
    \begin{cases}
        G_7(z_1^{s_n}, z_2^{s_n}, p_1^{s_n}, p_2^{s_n}, h^{s_n}, \tau^{s_n}, s_n):=
        &(\triangle+\lambda_1^*r_1(x))z_1^{s_n} \\
        &+d_1\nabla\cdot[(\phi_*+s_nz_1^{s_n})\nabla(\cos{\omega}\phi_*+w_1(s))] \\
        &+d_1\nabla\cdot[(\cos{\omega}\phi_*+w_1(s))\nabla (\phi_*+s_nz_1^{s_n})]e^{-s_nh^{s_n}\tau^{s_n}} \\
        &+(\phi_*+s_nz_1^{s_n})[(\lambda_1'(0)+o(s))r_1(x)-h^{s_n}] \\
        &-\lambda_{1s}(\phi_*+s_nz_1^{s_n})[2a_{11}(\cos{\omega}\phi_*+w_1(s))+a_{12}(\sin{\omega}\psi_*+w_2(s))]\\ 
        &-((p_1^{s_n}+ip_2^{s_n})\psi_*+s_nz_2^{s_n})a_{12}\lambda_{1s}(\cos{\omega}\phi_*+w_1(s)) \\
        G_8(z_1^{s_n}, z_2^{s_n}, p_1^{s_n}, p_2^{s_n}, h^{s_n}, \tau^{s_n}, s_n):=
        &(\triangle+\lambda_2^*r_2(x))z_2^{s_n} \\      
        &+ d_2\nabla\cdot[((p_1^{s_n}+ip_2^{s_n}) \psi_*+s_nz_2^{s_n}) \nabla(\sin{\omega}\psi_*+w_2(s))] \\
        &+d_2\nabla\cdot[(\sin{\omega}\psi_*+w_2(s))\nabla ((p_1^{s_n}+ip_2^{s_n})\psi_*+s_nz_2^{s_n})]e^{-s_nh^{s_n}\tau^{s_n}}  \\
        &-(\phi_*+s_nz_1^{s_n})a_{21}\lambda_{2s}(\sin{\omega}\psi_*+w_2(s)) \\
        &+\lambda_{2s}((p_1^{s_n}+ip_2^{s_n})\psi_*+s_nz_2^{s_n})[(\lambda_2'(0)+o(s))r_2(x)-h^{s_n}] \\
        &-((p_1^{s_n}+ip_2^{s_n})\psi_*+s_nz_2^{s_n})[a_{21}(\cos{\omega}\phi_*+w_1(s))\\
        &+2a_{22}(\sin{\omega}\psi_*+w_2(s))]\\   
        G_9(z_1^{s_n}, z_2^{s_n}, p_1^{s_n}, p_2^{s_n}, h^{s_n}, \tau^{s_n}, s_n):=
        &((p_1^{s_n})^2+(p_2^{s_n})^2-1)\|\psi_*\|^2_{Y_{\mathbb C}}+s_n^2\|z_1^{s_n}\|^2_{Y_{\mathbb C}}.
    \end{cases}
	\end{equation}
	and a continuously differentiable mapping from $(X_1)_{\mathbb C}^2\times\mathbb R^+\times \mathbb R^3$ to $Y_{\mathbb C}^2\times \mathbb R$ with 
	\[
	\hat{G}(z_1^{s_n}, z_2^{s_n}, p_1^{s_n}, p_2^{s_n}, h^{s_n}, \tau^{s_n}, s_n)=(G_7, G_8, G_9)^T=0.
	\]
	It is clear to see that $\{z_1^{s_n}\}$, $\{z_2^{s_n}\}$ are bounded in $Y_{\mathbb C}$ and $X_{\mathbb C}$, $\{p_1^{s_n}\}$, $\{p_2^{s_n}\}$ and $\{h^{s_n}\}$ are bounded for $s\in(0,\tilde\delta)$, which means that $\{(z_1^{s_n}, z_2^{s_n}, p_1^{s_n}, p_2^{s_n}, h^{s_n}, e^{-s_n\mathrm{Re}(h^{s_n})\tau^{s_n}}, e^{-s_n\mathrm{Im}(h^{s_n})\tau^{s_n}})\}$ is precompact in $Y_{\mathbb C}^2\times\mathbb R^5$. Thus, there also exists a subsequence $\{ (z_1^{s_{n_k}}, z_2^{s_{n_k}}, p_1^{s_{n_k}}, p_2^{s_{n_k}}, h^{s_{n_k}}, e^{-s_{n_k}\mathrm{Re}(h^{s_{n_k}})\tau^{s_{n_k}}}, \\e^{-s_{n_k}\mathrm{Im}(h^{s_{n_k}})\tau^{s_{n_k}}})\}_{k=1}^{\infty}$ satisfies $s_{n_k}\to 0$ as $k\to\infty$, and
	\[
	(z_1^{s_{n_k}}, z_2^{s_{n_k}}, p_1^{s_{n_k}}, p_2^{s_{n_k}}, h^{s_{n_k}}, e^{-s_{n_k}\mathrm{Re}(h^{s_{n_k}})\tau^{s_{n_k}}}, e^{-s_{n_k}\mathrm{Im}(h^{s_{n_k}})\tau^{s_{n_k}}})\to (\hat z_1, \hat z_2, \hat p_1, \hat p_2, \hat h, \hat\rho, e^{-i\hat\theta}),
	\]
	where $(\hat z_1, \hat z_2, \hat p_1, \hat p_2, \hat h, \hat\rho, \hat\theta)\in Y_{\mathbb C}^2\times\mathbb R^2\times\mathbb C\times [0,1]\times [0,2\pi]$ with $\mathrm{Re}(\hat h)\ge 0$. Then, we take $k\to\infty$ on the equation $\mathcal N(\lambda_*)^{-1}\hat G(z_1^{s_{n_k}}, z_2^{s_{n_k}}, p_1^{s_{n_k}}, p_2^{s_{n_k}}, h^{s_{n_k}}, \tau^{s_{n_k}}, s_{n_k})=0$, and obtain $(\hat z_1, \hat z_2)\in (X_1)_{\mathbb C}^2$ satisfying
	\begin{equation}\label{3.66}
    \begin{cases}
        0=& (\triangle+\lambda_1^*r_1(x))\hat z_1+ d_1[\nabla\cdot(\phi_*\nabla(\cos{\omega}\phi_*))+\nabla\cdot(\cos{\omega}\phi_*\nabla \phi_*)\hat\rho e^{-i\hat\theta}] \\
        &+\phi_*(\lambda_1'(0)r_1(x)-\hat h)-\lambda_1^*(2a_{11}\cos{\omega}\phi_*^2+a_{12}\sin{\omega}\phi_*\psi_*)-(\hat p_1+i\hat p_2)\lambda_1^*a_{12}\cos{\omega}\phi_*\psi_* , \\
        0=& (\triangle+\lambda_2^*r_2(x))\hat z_2+ d_2(\hat p_1+i\hat p_2)[\nabla\cdot(\psi_*\nabla(\sin{\omega}\psi_*))+\nabla\cdot(\sin{\omega}\psi_*\nabla \psi_*)\hat\rho e^{-i\hat\theta}] \\
        &+(\hat p_1+i\hat p_2)\psi_*(\lambda_2'(0)r_2(x)-\hat h)-a_{21}\lambda_2^*\sin{\omega}\phi_*\psi_*-\lambda_2^*(\hat p_1+i\hat p_2)(a_{21}\cos{\omega}\phi_*\psi_*+2a_{22}\sin{\omega}\psi_*^2) , \\
        0=& (\hat p_1)^2+(\hat p_2)^2-1.
    \end{cases}
	\end{equation}
	Taking the inner products of the first and the second equations of \eqref{3.66} with $\phi_*$ and $\psi_*$, respectively, gives that
	\begin{equation}\label{3.67}
    \begin{cases}
        d_1\kappa_1\hat\rho e^{-i\hat\theta}-K_1-K_4(\hat p_1+i\hat p_2)=\hat h, \\
        d_2(\hat p_1+i\hat p_2)\kappa_2\hat\rho e^{-i\hat\theta}-K_2(\hat p_1+i\hat p_2)-K_3 =(\hat p_1+i\hat p_2)\hat h.
    \end{cases}
    \end{equation}
    Separating the real and imaginary parts of two equations for \eqref{3.36}, we obtain the real parts
    \begin{equation}\label{3.68}
    \begin{cases}
        d_1\kappa_1\hat\rho\cos{\hat\theta}-K_1-K_4p_1^*=\mathrm{Re}(h^*), \\
        d_2\kappa_2\hat\rho\cos{\hat\theta}-K_2-K_3p_1^*=\mathrm{Re}(h^*).
    \end{cases}
    \end{equation}
    Hence, 
    \[
    \mathrm{Re}(h^*)=\dfrac{(d_1\kappa_1+d_2\kappa_2)\hat\rho\cos{\hat\theta}-(K_1+K_2)-(K_3+K_4)\hat p_1}{2}<0,
    \]
    which is a contradiction to $\mathrm{Re}(h^*)\ge 0$. The proof is completed.
\end{proof}

Based on the above discussion, we conclude the following results associated with the stability of the positive steady state $(u_s, v_s)$ and the Hopf bifurcation near $(u_s, v_s)$.
\begin{theorem}
	Under the assumptions $\mathbf{(H_0)}$ and $\mathbf{(H_1)}$, the following results hold:
	\begin{enumerate}
		\item[(i)] If $\mathbf{(H_2)}$ and $\mathbf{(H_5)}$holds, then the positive steady-state solution $(u_s, v_s)$ is locally asymptotically stable when $\tau\in[0, \tau_0)$, and $(u_s, v_s)$ is unstable when $\tau\in(\tau_0, \infty)$, $\tau_0$ is defined in \eqref{3.49}. Moreover, a Hopf bifurcation occurs at $(u_s, v_s)$ when $\tau=\tau_n$, $n=0, 1, 2, \dots$.
		\item[(i)] If $\mathbf{(H_3)}$ and $\mathbf{(H_6)}$holds, then the positive steady-state solution $(u_s, v_s)$ is locally asymptotically stable when $\tilde \tau\in[0, \tilde \tau_0)$, and $(u_s, v_s)$ is unstable when $\tau\in(\tilde \tau_0, \infty)$, $\tilde \tau_0$ is defined in \eqref{60}. Moreover, a Hopf bifurcation occurs at $(u_s, v_s)$ when $\tau=\tilde \tau_n$, $n=0, 1, 2, \dots$.
		\item[(iii)] If $(d_1, d_2)\in D_2$, then the positive steady-state solution $(u_s, v_s)$ is locally asymptotically stable when $\tau\ge 0$.
	\end{enumerate}
\end{theorem}

\begin{figure}[htb]
  \centering
  \includegraphics[width=0.35\textwidth]{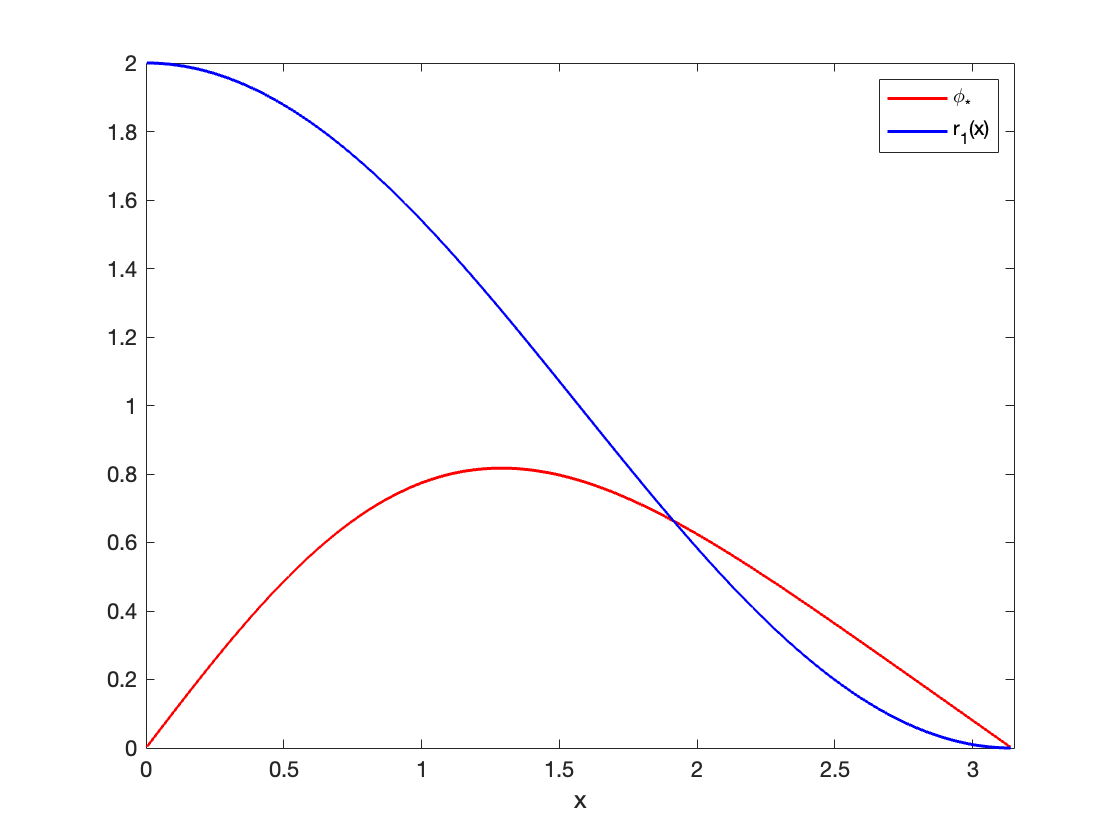}
  \includegraphics[width=0.35\textwidth]
  {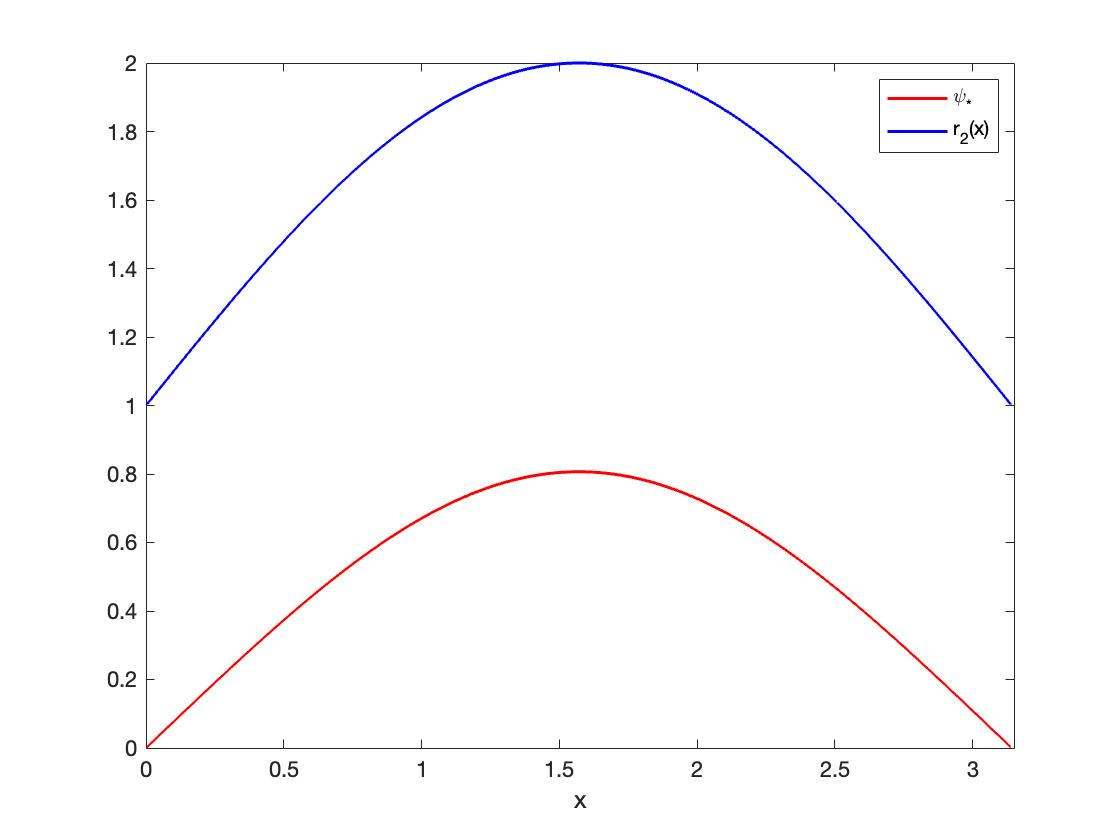}
  \caption{Left: the resource distributed function $r_1(x)=\cos{x}+1$, and the principle eigenvalue function $\phi_*(x)$ of \eqref{2.3}; Right: the resource distributed function $r_2(x)=\sin{x}+1$, and the principle eigenvalue function $\psi_*(x)$ of \eqref{2.4} }
  \label{fig_r}
\end{figure}
\section{Numerical simulations}
In this section, we show numerical simulations to verify our results in sections 2-3. We consider a spatial domain $\Omega=(0,\pi)$. We choose the functions:
\[
r_1(x)=\cos{x}+1, \quad r_2(x)=\sin{x}+1,
\]
onto the eigenvalue problems \eqref{2.3} and \eqref{2.4}. And then, the principle eigenvalues are $\lambda_1^*=0.9291$, $\lambda_2^*=0.5403$ and their corresponding eigenvalue functions $\phi_*(x)$ and $\psi_*(x)$ in Figure 2, respectively.

\begin{figure}[htbp]
\centering
\begin{tikzpicture}
    \begin{axis}[
        axis lines=middle,
        xmin=-2, xmax=2,
        ymin=-2, ymax=2,
        width=8cm, height=8cm,
        grid=none,
        xtick=\empty, 
        ytick=\empty,
        xlabel={$d_1$}, 
        ylabel={$d_2$},
        domain=-3:3, 
        enlargelimits=true, 
        samples=100 
    ]  
    \addplot [
            pattern=north west lines, 
            pattern color=olive, 
            opacity=0.5, 
        ]
        coordinates {(1/4, 3/4) (5/12, 7/12) (5/3, 5/6) (3, 1.5) (3, 3) (1, 3)} ;
    \addplot [
            pattern=north west lines, 
            pattern color=olive, 
            opacity=0.5, 
        ]
        coordinates {(-1/4, -3/4) (-5/12, -7/12) (-5/3, -5/6) (-3, -1.5) (-3, -3) (-1, -3)} ;  
    \addplot [
    		pattern=north east lines, 
            pattern color=c1, 
            opacity=0.5, 
        ]
        coordinates {(3,-2) (3,-3) (2,-3) (-3,2) (-3,3) (-2,3)} ;   
        \addplot[gray, thick] {-1*x+1};
        \addplot[gray, thick] {-1*x-1};   
        \addplot[gray, thick] {0.5*x};
        \addplot[gray, thick] {3*x};
        \addplot[gray, thick] {0.2*x + 0.5};
        \addplot[gray, thick] {0.2*x - 0.5};
        \node at (axis cs:-2, -0.5) [fill=none, text=black] {$D_1$};
        \node at (axis cs:-2, 2) [fill=none, text=black] {$D_2$};
        \node at (axis cs:2, 2) [fill=none, text=black] {$\mathbf{(H_5)}$};
        \node at (axis cs:-2, -2) [fill=none, text=black] {$\mathbf{(H_5)}$};
        \fill[red] (axis cs:1, -1) circle (1.5pt);
        \node at (axis cs:1, -1) [anchor=south east, fill=none, text=black] {$P_1$};
        \fill[red] (axis cs:0.2, 0.3) circle (1.5pt);
        \node at (axis cs:0.2, 0.3) [anchor=south east, fill=none, text=black] {$P_2$};
        \fill[red] (axis cs:1, 1.5) circle (1.5pt);
        \node at (axis cs:1, 1) [anchor=south east, fill=none, text=black] {$P_3$};
    \end{axis}
    \node at (4.5, 1) [fill=none, text=black] {$l_1$};
    \node at (6, 1) [fill=none, text=black] {$l_2$};
    \node at (0, 2) [fill=none, text=black] {$l_3$};
	\node at (0, 3) [fill=none, text=black] {$l_4$};
    \node at (6, 5) [fill=none, text=black] {$l_5$};
	\node at (4.5, 6) [fill=none, text=black] {$l_6$};
\end{tikzpicture}
\caption{The regions of $(d_1, d_2)$ satisfying $\mathbf{(H_5)}$ based on Figure \ref{f1.1}, where $l_5: d_2=\frac{K_2\kappa_1}{K_1\kappa_2}d_1$, $l_6: d_2=\frac{(K_2-K_3)\kappa_1}{(K_1-K_4)\kappa_2}d_1$.}
\label{f1.2}
\end{figure}
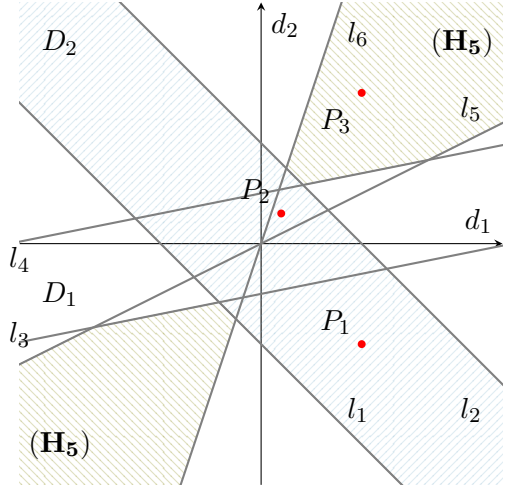

we first choose the following parameters satisfying $(\mathbf{H_2})$,
\begin{enumerate}
	\item[$(\mathbf{Q_1})$] $a_{11}=0.5$, $a_{12}=0.5$, $a_{21}=1$, $a_{22}=1.5$, $\omega=\dfrac{\pi}{4}$, $\lambda_1=2$, $\lambda_2=2$.
\end{enumerate}

In Figure \ref{fig_P_1_2}, we calculate and obtain that
\[
l_1: d_2=-1.0622d_1-0.9050, \quad
l_2: d_2=-1.0622d_1+0.9050, \quad
l_5: d_2=1.8466d_1, 
\]
\[
l_3: d_2=1.2379d_1-0.1827, \quad
l_4: d_2=1.2379d_1+0.1827, \quad
l_6: d_2=30.0015d_1.
\]
Then, we select two points $P_1=(1,-1)$, $P_2=(0.1,0.5)$ in the region $D_2$ and $P_3=(1,3)$. It is worth noting that if both $d_1>0$ and $d_2>0$ are small, then the positive steady state $(u_s, v_s)$ is locally asymptotically stable for $\tau\ge 0$; but if $d_1>0$ and $d_2>0$ are relatively large, then the stability of the positive steady state $(u_s, v_s)$ will be changed, there exists a critical value $\tau_0$, such that $(u_s, v_s)$ is locally asymptotically stable for $\tau\in[0,\tau_0]$ and loses the stability for $\tau\in[\tau_0, \infty]$, which implies that a family of spatially non-homogeneous time-periodic solutions occur, shown in Figure \ref{fig_P_1_2} and Figure \ref{fig_P_3}.

\begin{figure}[htb]
	\centering
	\subfloat[$P_1(d_1, d_2)=(1,-1)\in D_2$ with $\tau=10$]{
	\includegraphics[width=0.35\linewidth]{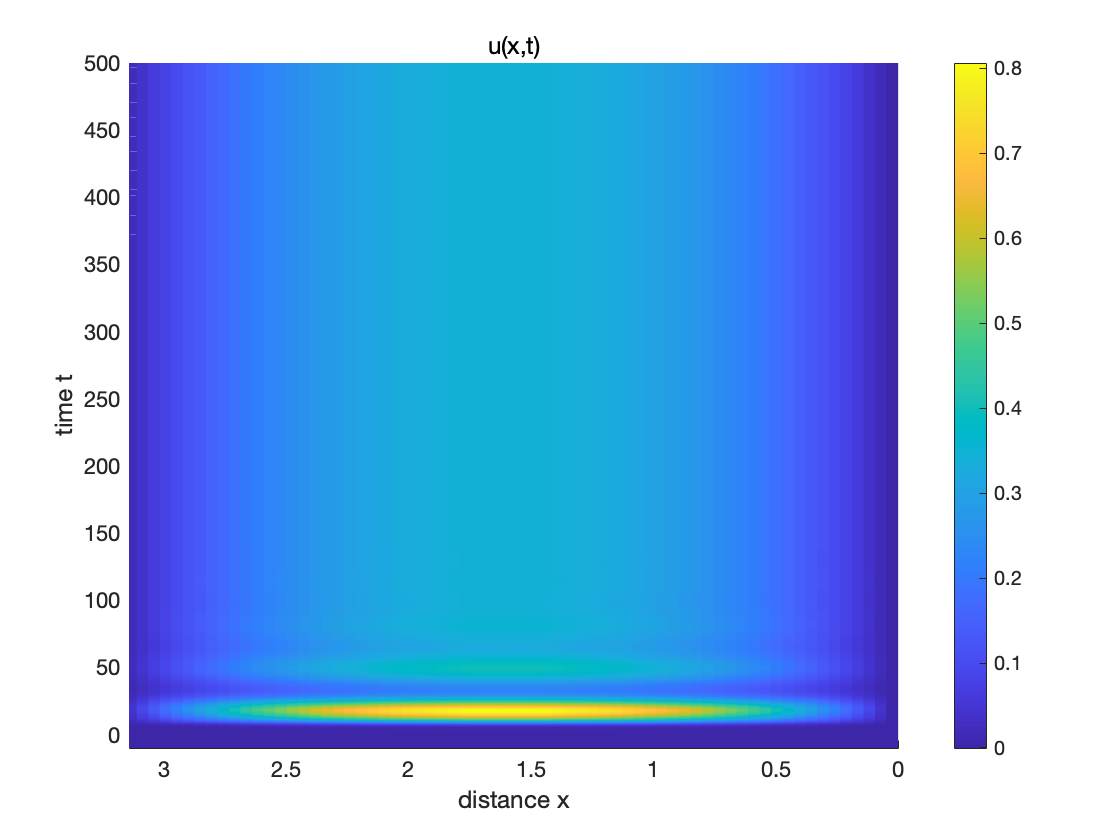}
	\includegraphics[width=0.35\textwidth]{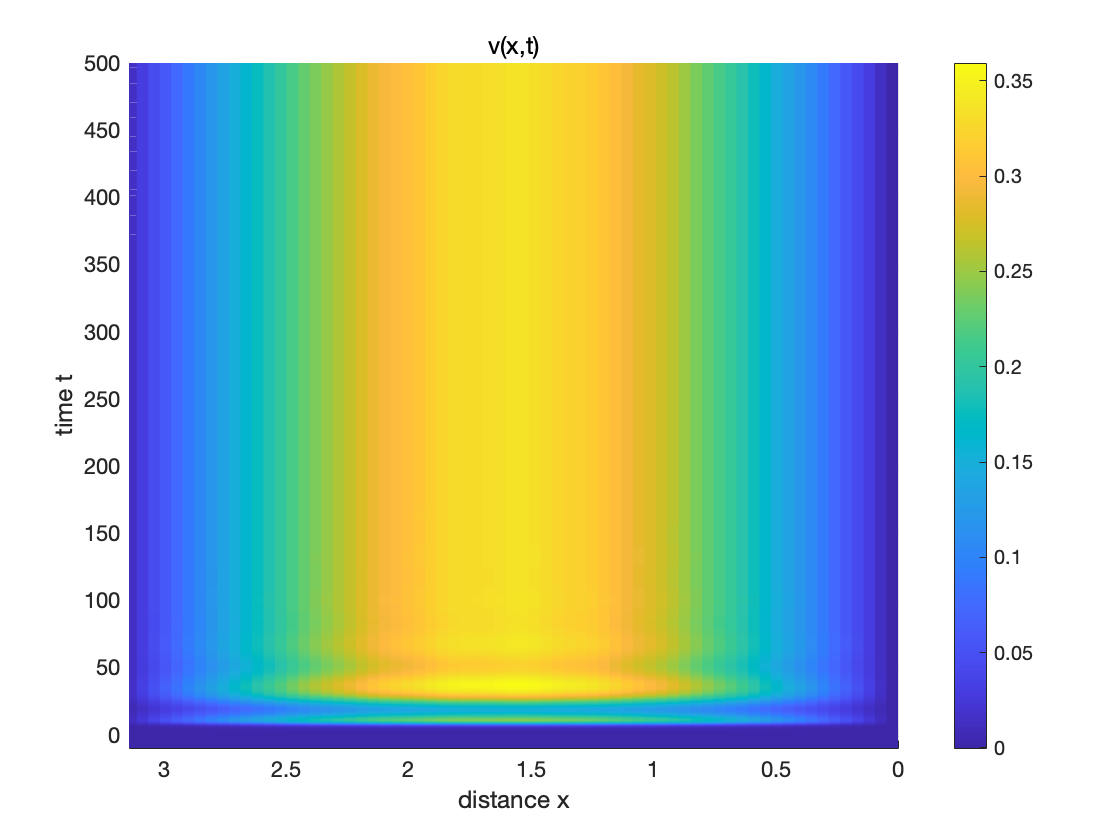}
	}
	\hfill
	\subfloat[$P_2(d_1, d_2)=(0.1,0.5)\in D_2$ with $\tau=10$]{
	\includegraphics[width=0.35\linewidth]{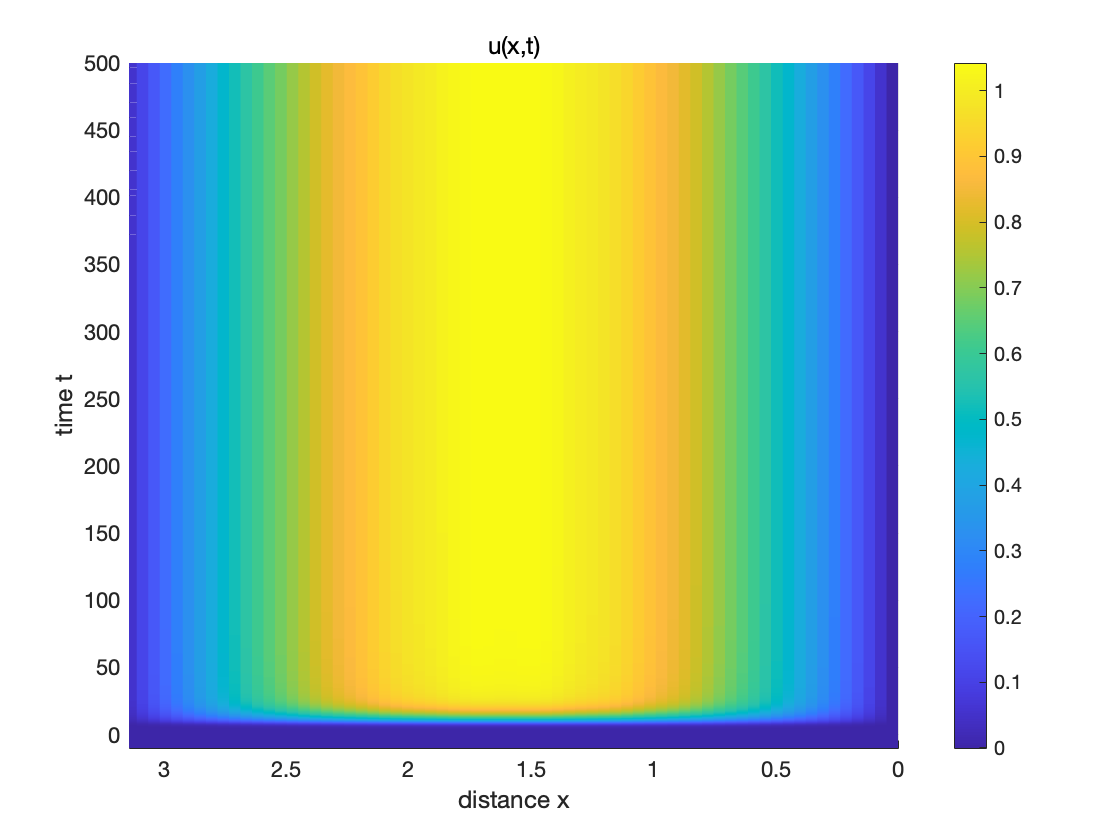}
	\includegraphics[width=0.35\textwidth]{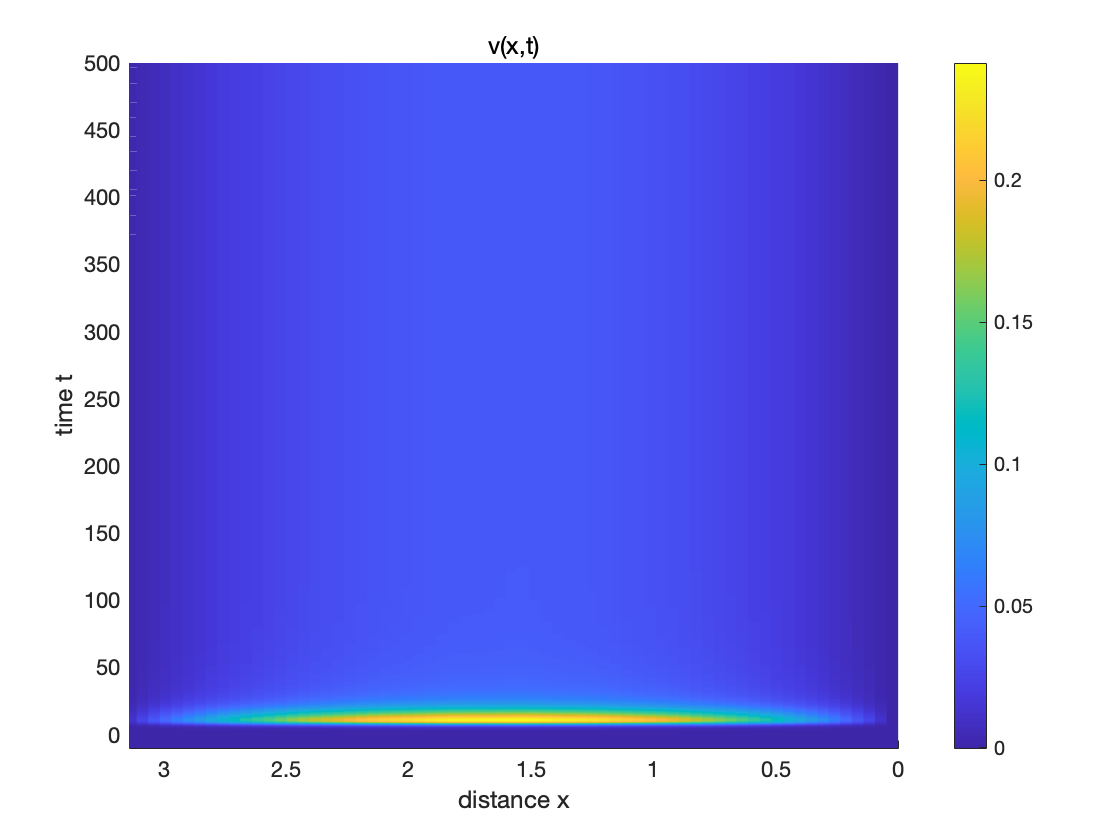}
	}
	\hfill
	\caption{(a) Numerical simulations of \eqref{1} for $P_1\in D_2\cap (\mathbb R^+\times\mathbb R^-)$, the positive steady state $(u_s, v_s)$ is locally asymptotically stable for $\tau\ge 0$; (b) Numerical simulations of \eqref{2.6} for $P_2\in D_2\cap (\mathbb R^+\times\mathbb R^+)$, the positive steady state $(u_s, v_s)$ is locally asymptotically stable for $\tau\ge 0$.}
  \label{fig_P_1_2}
\end{figure}

\begin{figure}[htb]
  \centering
  \subfloat[$P_3(d_1, d_2)=(1,3)\in D_3$ satisfying $\mathbf{(H_5)}$ with $\tau=4$]{
  \includegraphics[width=0.35\linewidth]{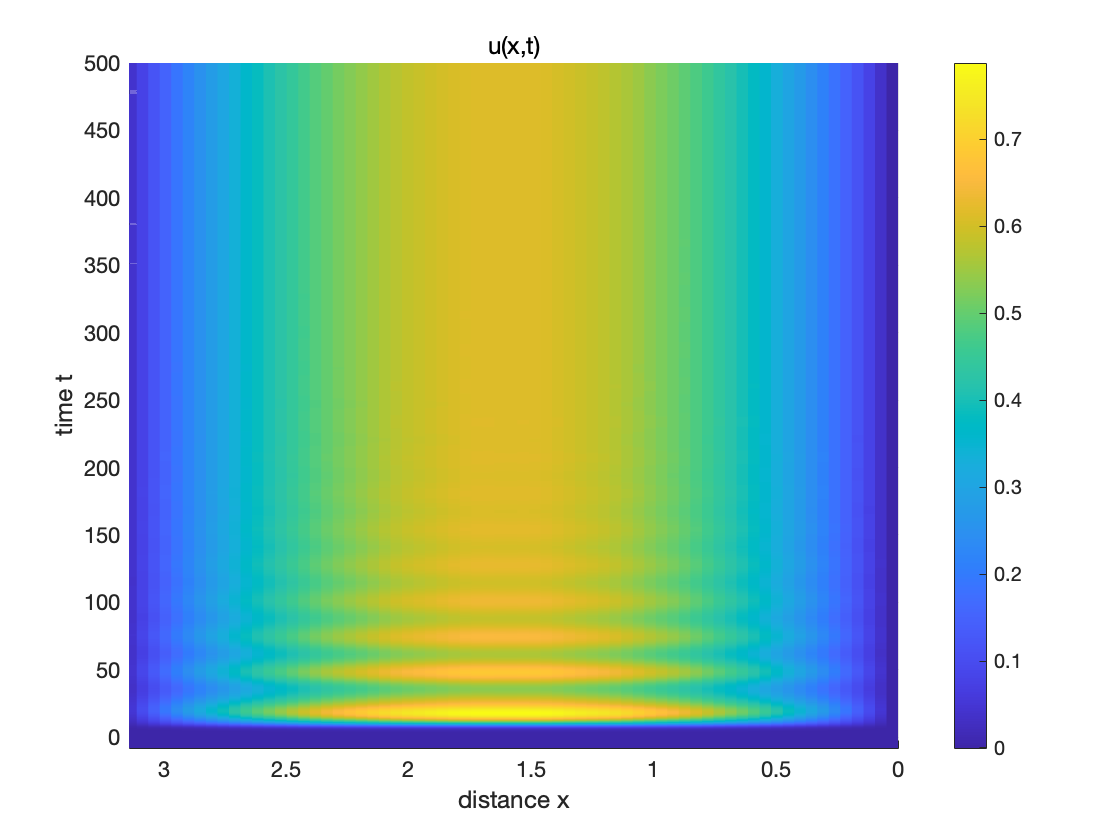}
  \includegraphics[width=0.35\textwidth]{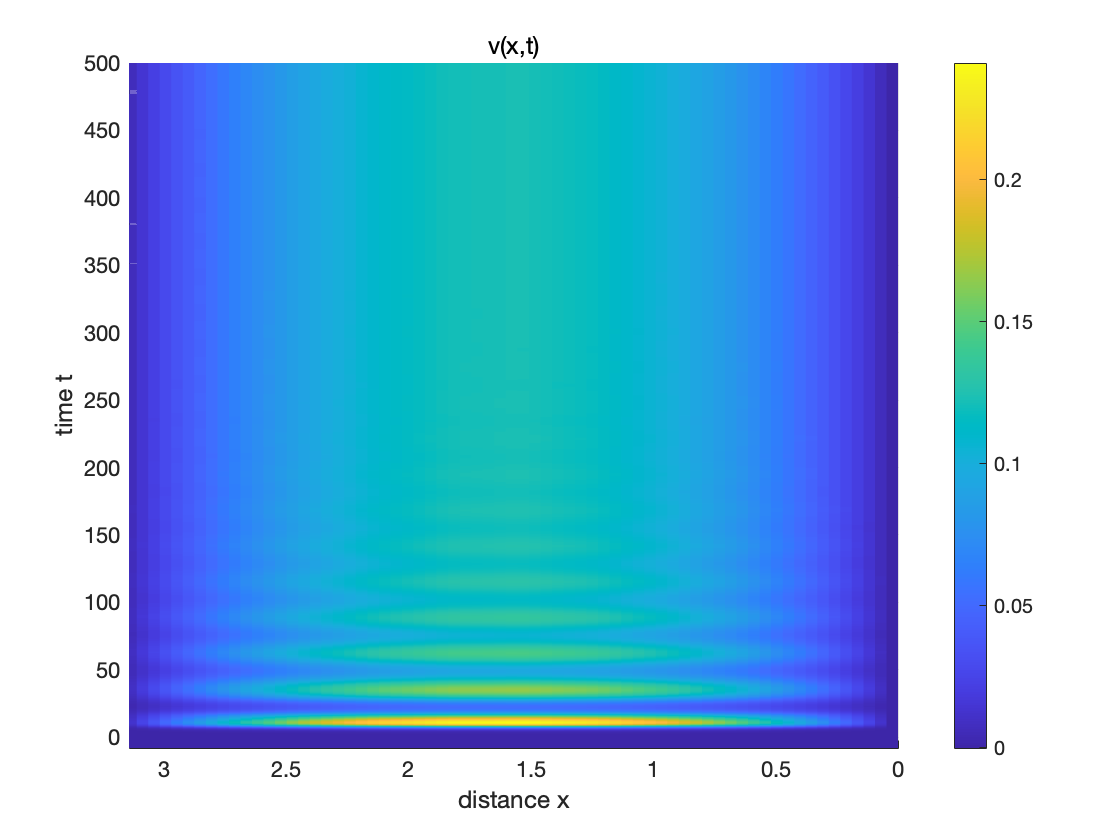}
  }
	\hfill
  \subfloat[$P_3(d_1, d_2)=(1,3)\in D_3$ satisfying $\mathbf{(H_5)}$ with $\tau=10$]{
  \includegraphics[width=0.35\linewidth]{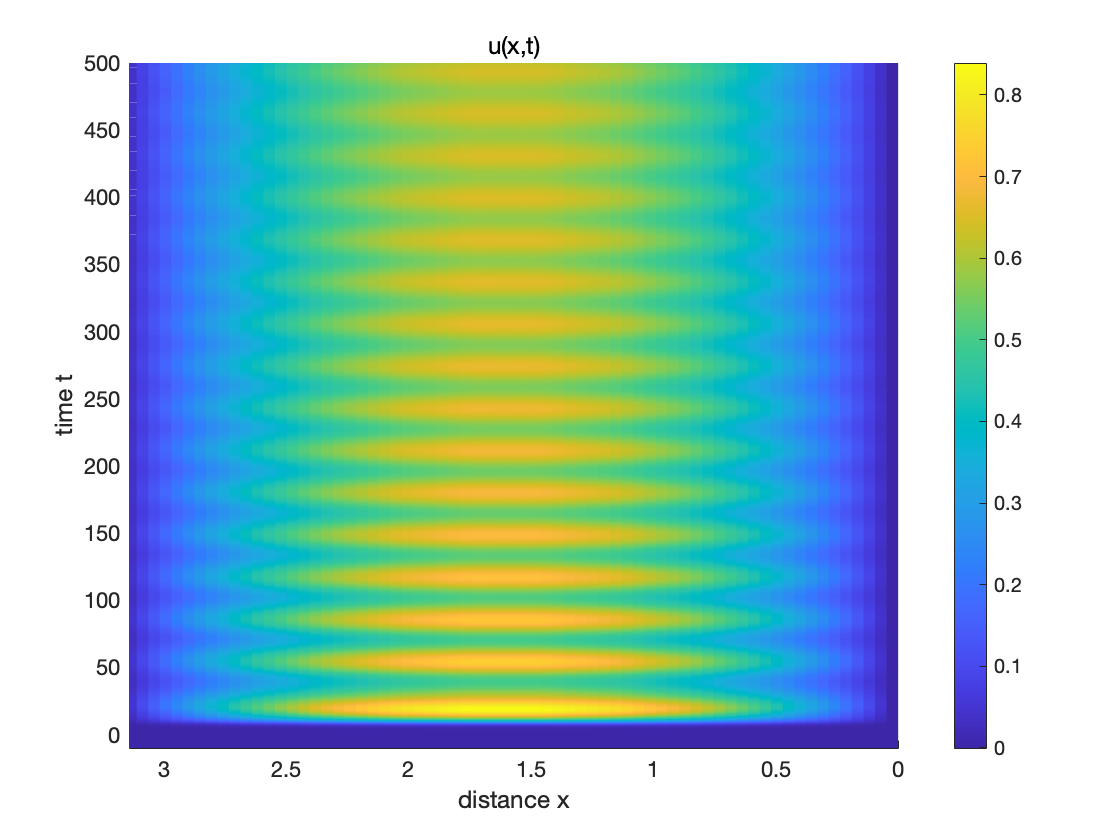}
  \includegraphics[width=0.35\textwidth]{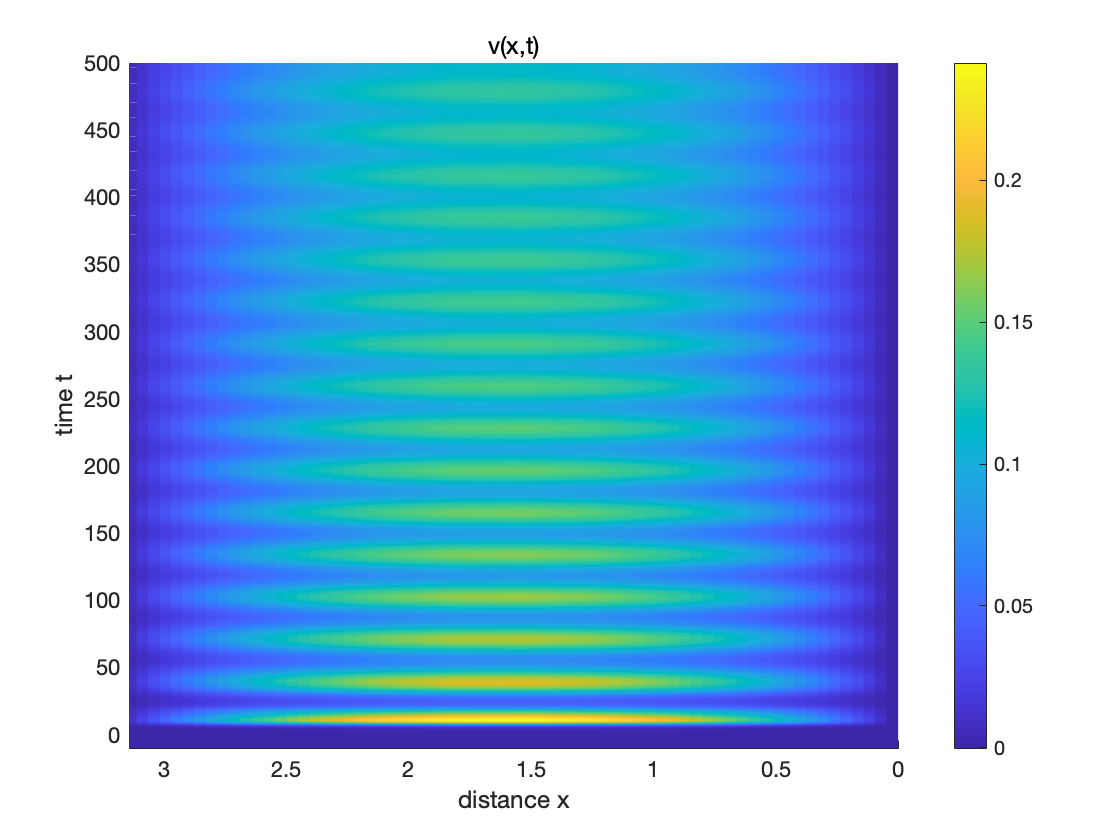}
  }
	\hfill
  \caption{Numerical simulations of \eqref{2.6} for $P_3\in D_3$ satisfying $\mathbf{(H_5)}$, (a) $\tau=4<\tau_0=4.6458$, the positive steady state $(u_s, v_s)$ is locally asymptotically stable; (b) $\tau=10>\tau_0=4.6458$, the positive steady state $(u_s, v_s)$ converges to a spatially non-homogeneous time-periodic solution.}
  \label{fig_P_3}
\end{figure}

Next, we consider the dynamics of weak competition and choose the following parameters satisfying $(\mathbf{H_3})$,
\begin{enumerate}
	\item[$(\mathbf{Q_2})$] $a_{11}=1$, $a_{12}=0.5$, $a_{21}=0.8$, $a_{22}=1$, $\omega=\dfrac{\pi}{4}$, $\lambda_1=2$, $\lambda_2=2$.
\end{enumerate}

\begin{figure}[htbp]
\centering
\begin{tikzpicture}
    \begin{axis}[
        axis lines=middle,
        xmin=-2, xmax=2,
        ymin=-2, ymax=2,
        width=8cm, height=8cm,
        grid=none,
        xtick=\empty, 
        ytick=\empty,
        xlabel={$d_1$}, 
        ylabel={$d_2$},
        domain=-3:3, 
        enlargelimits=true, 
        samples=100 
    ]  
    \addplot [
            pattern=north west lines, 
            pattern color=olive, 
            opacity=0.6, 
        ]
        coordinates {(1, 1) (2/3, 1/3) (5/6, 1/6) (3, 0.6) (3, 3) (3, 3)} ;
    \addplot [
            pattern=north west lines, 
            pattern color=olive, 
            opacity=0.6, 
        ]
        coordinates {(-1, -1) (-2/3, -1/3) (-5/6, -1/6) (-3, -0.6) (-3, -3) (-3, -3)} ;
    \addplot [
    		pattern=north east lines, 
            pattern color=c1, 
            opacity=0.5, 
        ]
        coordinates {(3,-2) (3,-3) (2,-3) (-3,2) (-3,3) (-2,3)} ;   
        \addplot[gray, thick] {-1*x+1};
        \addplot[gray, thick] {-1*x-1};    
        \addplot[gray, thick] {x};
        \addplot[gray, thick] {0.2*x};
        \addplot[gray, thick] {2*x + 1};
        \addplot[gray, thick] {2*x - 1}; 
        \node at (axis cs:-1, -2) [fill=none, text=black] {$D_1$};
        \node at (axis cs:-2, 2) [fill=none, text=black] {$D_2$};
        \node at (axis cs:2, 1) [fill=none, text=black] {$\mathbf{(H_6)}$};
        \node at (axis cs:-2, -1) [fill=none, text=black] {$\mathbf{(H_6)}$};
        \fill[red] (axis cs:1.5, 0.7) circle (1.5pt);
        \node at (axis cs:1.5, 0.7) [anchor=south east, fill=none, text=black] {$P_4$};
    \end{axis}
    \node at (4.5, 1) [fill=none, text=black] {$l_1$};
    \node at (6, 1) [fill=none, text=black] {$l_2$};
    \node at (2, 0) [fill=none, text=black] {$l_3$};
	\node at (1, 0) [fill=none, text=black] {$l_4$};
    \node at (6, 4) [fill=none, text=black] {$l_5$};
	\node at (6, 6) [fill=none, text=black] {$l_6$};
\end{tikzpicture}
\caption{The regions of $(d_1, d_2)$ satisfying $\mathbf{(H_6)}$ based on Figure \ref{f1.1}, where $l_5: d_2=\frac{K_2\kappa_1}{K_1\kappa_2}d_1$, $l_6: d_2=\frac{(K_2+K_3)\kappa_1}{(K_1+K_4)\kappa_2}d_1$.}
\label{f1.3}
\end{figure}
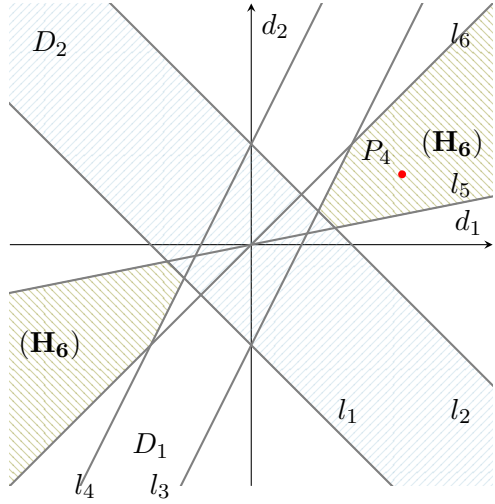

By calculating, we obtain 
\[
l_1: d_2=-1.0622d_1-1.0101, \quad
l_2: d_2=-1.0622d_1+1.0101, \quad
l_5: d_2=0.6155d_1, 
\]
\[
l_3: d_2=0.9903d_1-0.2249, \quad
l_4: d_2=0.9903d_1+0.2249, \quad
l_6: d_2=0.7387d_1,
\]
in Figure \ref{fig_P_4}, we select the point $P_4=(2,1.4)$, it is similar to the case $(\mathbf{H_2})$ for the results. There are two results worth noting, one being that when $d_1>0$ and $d_2>0$ are small, the positive steady state $(u_s, v_s)$ is locally asymptotically stable for any $\tau\ge 0$, but when $d_1$ and $d_2$ are relatively large and their ratio keep in a range, the positive steady state $(u_s, v_s)$ will lose its stability, Hopf bifurcation occurs, shown in Figure. 7. Another is that when $r_1(x)\neq r_2(x)$, there might Hopf bifurcation occurs when $\tau$ increases under weak competition, which is different from the result in Wang et al. \cite{WW2023}, they show that when $r_1(x)= r_2(x)$, then the positive steady state is always locally asymptotically stable for $\tau\ge 0$.

\begin{figure}[htbp]
  \centering
  \subfloat[$P_4(d_1, d_2)=(2,1.4)\in D_3$ satisfying $\mathbf{(H_6)}$ with $\tau=3$]{
  \includegraphics[width=0.35\linewidth]{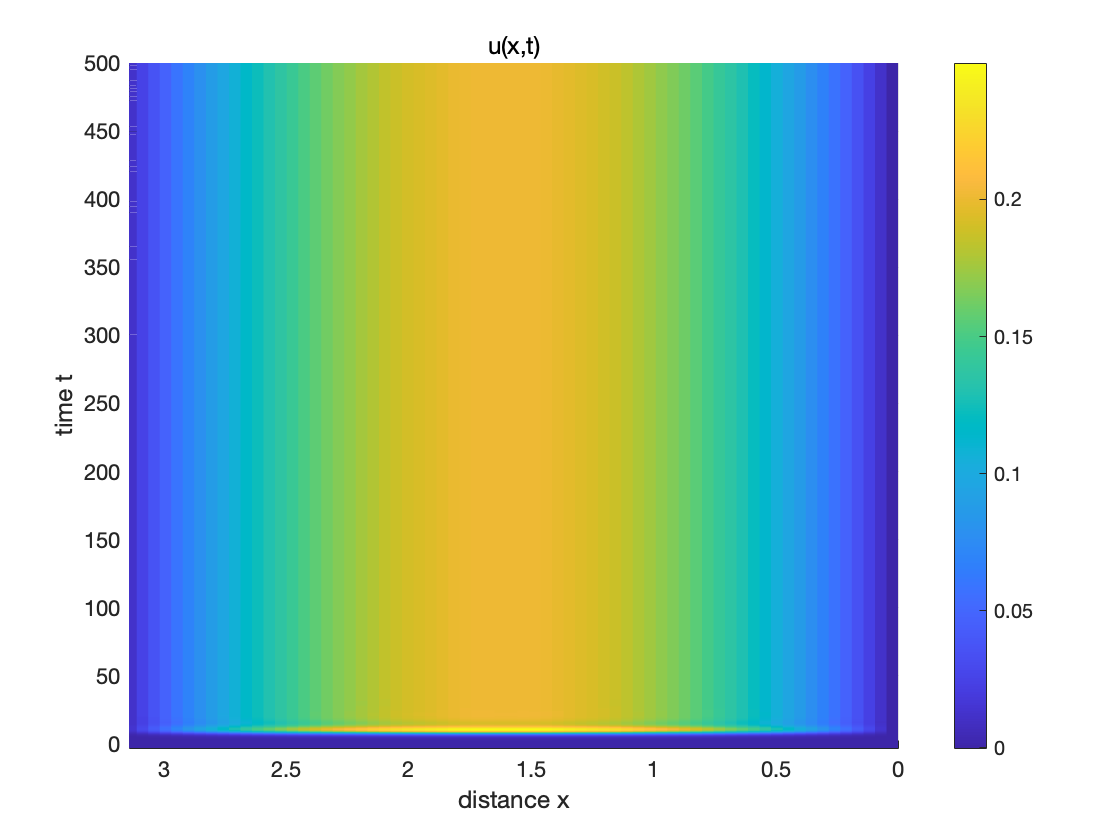}
  \includegraphics[width=0.35\textwidth]{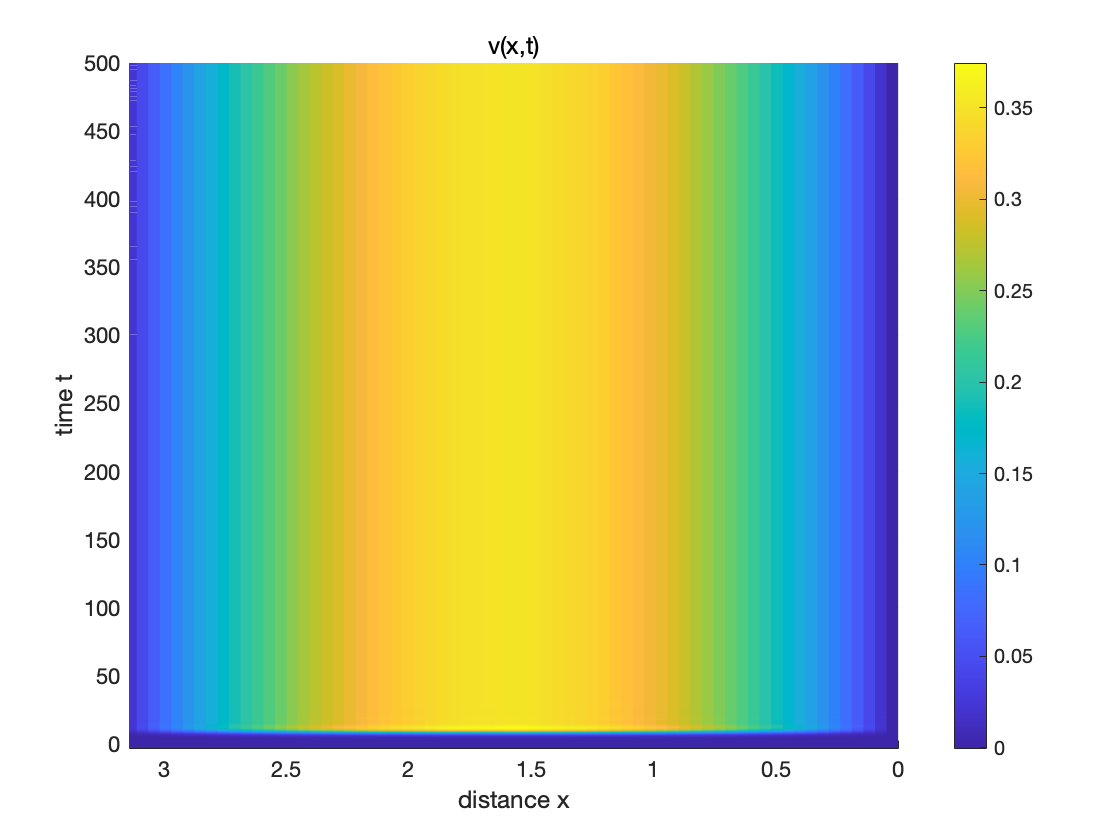}
  }
	\hfill
  \subfloat[$P_4(d_1, d_2)=(2,1.4)\in D_3$ satisfying $\mathbf{(H_6)}$ with $\tau=17$]{
  \includegraphics[width=0.35\linewidth]{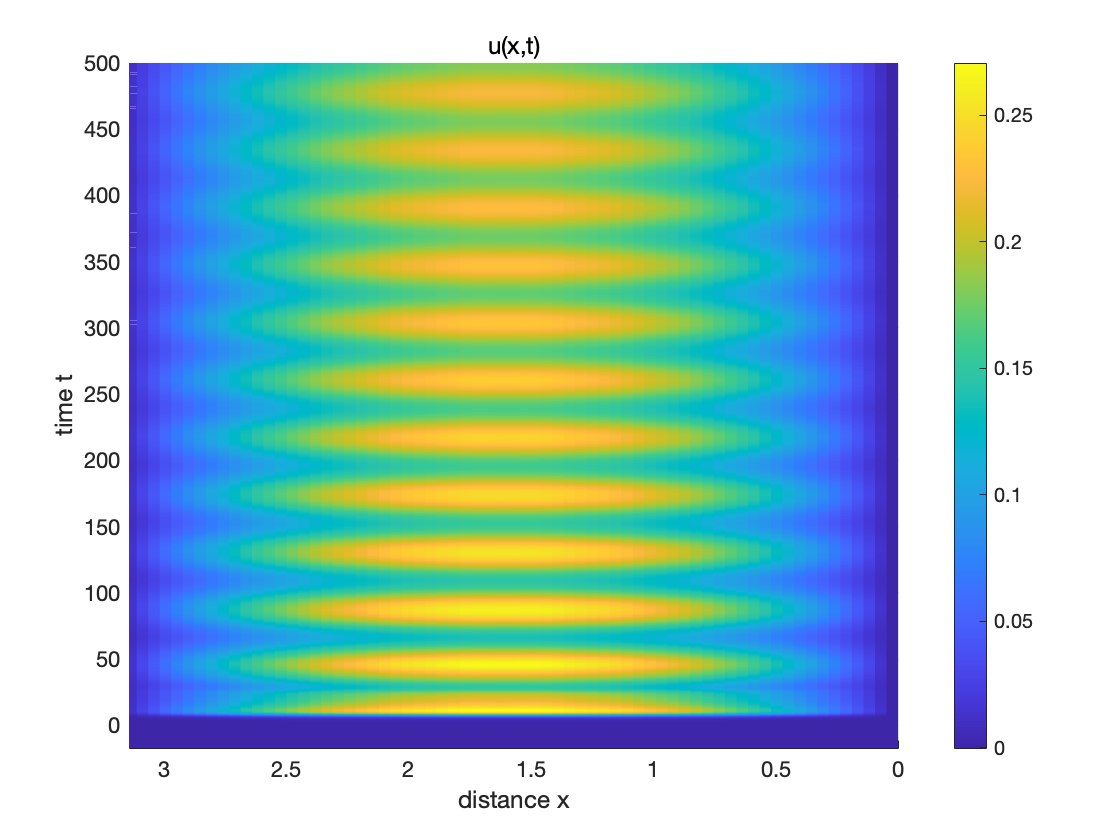}
  \includegraphics[width=0.35\textwidth]{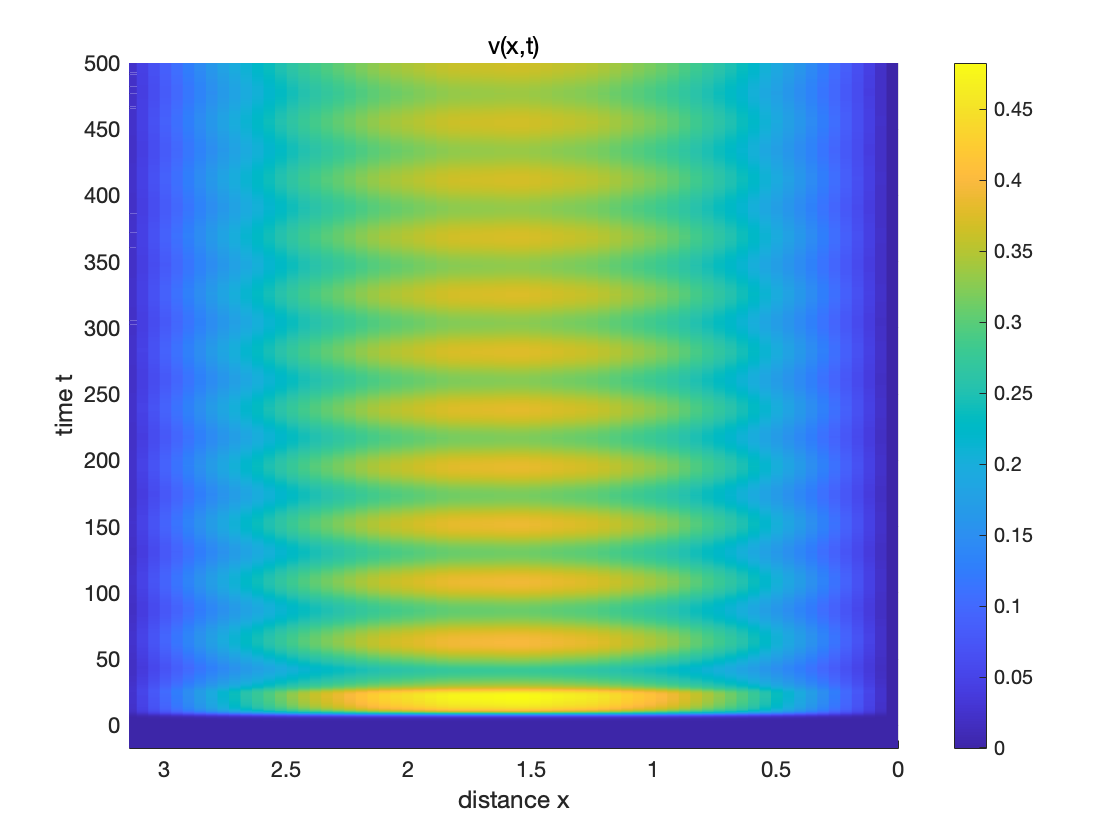}
  }
	\hfill
  \caption{Numerical simulations of \eqref{2.6} for $P_4\in D_3$ satisfying $\mathbf{(H_6)}$, (a) $\tau=3<\tilde \tau_0=3.3592$, the positive steady state $(u_s, v_s)$ is locally asymptotically stable; (b) $\tau=17>\tilde \tau_0=3.3592$, the positive steady state $(u_s, v_s)$ converges to a spatially non-homogeneous time-periodic solution.}
  \label{fig_P_4}
\end{figure}

\section{Conclusion}

In this paper, we consider a Lotka-Volterra competition model with memory effect and spatial heterogeneity, whose two competing species have their self-memory, respectively. Unlike most literature on competition models with memory, the boundary condition we consider is the Dirichlet boundary condition. It can be challenging because the positive steady state must be spatially non-constant, and it is usually more difficult to analyze the spatiotemporal dynamics of systems with Dirichlet boundary conditions. Moreover, considering the heterogeneity of landscapes, animal memory plays a crucial role in determining how resources are efficiently utilized in space. A comprehensive consideration of memory, heterogeneous environments, and hostile boundary conditions in the dynamics of competitive populations could provide potential insights for wildlife and ecological conservation.

We mainly focus on the stability of the coexistence steady state and the existence of Hopf bifurcation for the system \eqref{1}. Moreover, some interesting conclusions have been observed in our study. One is that when $(d_1,d_2)\in D_2$ and $d_1d_2<0$, the positive non-constant steady state $(u_s, v_s)$ is locally asymptotically stable for $\tau\ge 0$. This means that if the movement direction of one competitive species is towards the gradient of its density function in the past time, while another avoids the gradient of its density function in the same past time, then stable coexistence of the two competitive species can occur when their memory diffusion coefficients $d_1$ and $d_2$ are within a certain range. On the other hand, When coefficients $d_1$, $d_2>0$ and they are both small, the positive non-constant steady state $(u_s, v_s)$ is also locally asymptotically stable for $\tau\ge 0$, which means that when both of two competing species move towards the gradient of density function in the past time, then their coexistence is stable with two small memory diffusion coefficients $d_1$ and $d_2$. However, when coefficients $d_1$, $d_2>0$ are relatively large, and their ratio keeps in a certain range, the stability of the positive steady state $(u_s, v_s)$ could be lost. there exists a critical value $\tau_0$, such that $(u_s, v_s)$ is locally asymptotically stable when $\tau\in[0,\tau_0]$, while unstable when $\tau>\tau_0$, the positive steady state $(u_s, v_s)$ converges to a spatially non-homogeneous time-periodic solution, Hopf bifurcation occurs. 

Moreover, another intriguing conclusion is related to the impact of spatial heterogeneity functions on the system dynamics. The best example is in the case of weak competition, where spatial heterogeneity functions can significantly influence the stability and coexistence of two competing species. When $r_1(x)\neq r_2(x)$, there might be a Hopf bifurcation occurring for $\tau$ increasing and two memory-based diffusion coefficients $(d_1, d_2)$ keep in a region. It is different from the result shown in Wang et al. \cite{WW2023}, that is, when $r_1(x)= r_2(x)$, then the positive non-constant steady state is locally asymptotically stable for $\tau\ge 0$.

Overall, this study demonstrates the complex role of memory effects, spatial heterogeneity, and Dirichlet boundary conditions in competitive species dynamics, and reveals the potential for spatial periodicity pattern formation in the model. Future research can further explore the interaction between memory effects and spatial heterogeneity in different ecological environments, aiming to provide theoretical support for conservation strategies and ecosystem management.

\end{document}